\documentclass[10pt]{amsart}
\usepackage{amsfonts,amssymb,amscd,amsmath,enumerate,verbatim,calc,amsthm}
\input xy
\xyoption{all}

\headsep=1cm \numberwithin{equation}{section}
\newtheorem{thm}{Theorem}[section]
\newtheorem{cor}[thm]{Corollary}
\newtheorem{no}[thm]{Notation}
\newtheorem{lem}[thm]{Lemma}
\newtheorem{prop}[thm]{Proposition}
\newtheorem{defn}[thm]{Definition}

\newtheorem{rem}[thm]{Remark}
\newtheorem{Ques}[thm]{Question}

\newcommand{\coker}{\mbox{Coker}\,}
\newcommand{\Hom}{\mbox{Hom}\,}
\newcommand{\Ext}{\mbox{Ext}\,}
\newcommand{\Tor}{\mbox{Tor}\,}

\newcommand{\Ass}{\mbox{Ass}\,}

\renewcommand{\dim}{\mbox{dim}\,}

\newcommand{\pd}{\mbox{pd}\,}
\newcommand{\id}{\mbox{id}\,}
\newcommand{\cid}{\mbox{cid}\,}
\newcommand{\cpd}{\mbox{cpd}\,}
\newcommand{\cfd}{\mbox{cfd}\,}
\newcommand{\ciD}{\mbox{ciD}\,}
\newcommand{\cpD}{\mbox{cpD}\,}
\newcommand{\cfD}{\mbox{cfD}\,}
\newcommand{\scd}{\mbox{scd}\,}
\newcommand{\std}{\mbox{std}\,}
\newcommand{\sci}{\mbox{sci}\,}
\newcommand{\scf}{\mbox{scf}\,}
\newcommand{\scp}{\mbox{scp}\,}
\newcommand{\ci}{\mbox{ci}\,}
\newcommand{\cf}{\mbox{cf}\,}
\newcommand{\cp}{\mbox{cp}\,}
\newcommand{\scot}{\mbox{scot}\,}
\newcommand{\stf}{\mbox{stf}\,}

\newcommand{\Cot}{\mbox{cot}\,}

\newcommand{\fd}{\mbox{fd}\,}

\newcommand{\mcm}{\mbox{MCM}\,}

\renewcommand{\H}{\mbox{H}}

\newcommand{\fm}{\mathfrak{m}}
\newcommand{\fp}{\mathfrak{p}}

\begin{document}
\bibliographystyle{amsplain}


\title[Subcategories and homological dimensions related to Semidualizing modules]
{Some new Subcategories and homological dimensions related to Semidualizing modules}

\bibliographystyle{amsplain}

     \author[M. Rahmani]{Mohammad Rahmani}
     \author[A.- J. Taherizadeh]{Abdoljavad Taherizadeh}

\address{Faculty of Mathematical Sciences and Computer,
Kharazmi University, Tehran, Iran.}

\email{m.rahmani.math@gmail.com}
\email{taheri@khu.ac.ir}



\begin{abstract}
Let $R$ be a Noetherian ring and let $C$ be a semidualizing $R$-module. In this paper, by using the semidualizing modules,
we define and study new classes of modules and homological dimensions and investigate the relations between them. In parallel, we obtain some necessary and sufficient condition
for $ C $ to be dualizing.
\end{abstract}

\maketitle

\bibliographystyle{amsplain}
\section{introduction}

Throughout this paper, $R$ is a commutative Noetherian ring with non-zero identity. A finitely generated $R$-module $C$ is semidualizing if the natural homothety map $ R\longrightarrow \Hom_R(C,C) $ is an isomorphism and $ \Ext^i_R(C,C)=0 $ for all $ i>0 $. Foxby \cite{F}, Vasconcelos \cite{V}
and Golod \cite{G} independently initiated the study of semidualizing modules. The most important classes of modules over a Noetherian ring $ R $,
from a homological point of view, are the projective, flat, and injective modules. In \cite{E2}, E. Enochs introduced the notion of cotorsion modules and J.Xu \cite{X}, generalized this notion and defined two classes of $R$-modules, namely strongly cotorsion and strongly torsionfree $R$-modules. The notion of G-injective, G-projective and G-flat modules has been defined by E. Enochs et al. in \cite{EJ2}, \cite{EJT}. It is a
refinement of the classical injective, projective and flat modules and shares some of their nice properties. By using the semidualizing modules in \cite{HJ}, as a generalization of G-injective, G-projective and G-flat modules, H. Holm and P. Jørgensen introduced the notion of
 G$ _C $-injective, G$ _C $-projective and G$ _C $-flat modules. It is well-known that over a finite dimensional Gorenstein ring, then the strongly
 cotorsion (resp. strongly torsion free) modules are precisely the G-injective (resp. G-flat) modules. In Section 3, we generalize this fact to a Cohen-Macaulay ring with dualizing module. More precisely, in Corollary 3.4, we prove the following: \\
\textbf{Theorem A.} Let $ R $ be a ring of finite Krull dimension $ d $ and let $C$ be dulalizing.
\begin{itemize}
	\item[(i)]{One has $ \mathcal{GI}_C = \mathcal{M}_{\scot} $.}
	\item[(ii)] {One has $ \mathcal{GF}_C = \mathcal{M}_{\stf} $.}
\end{itemize}

 In \cite{EJ3} and \cite{EJ4}, E. Enochs and O. Jenda introduced the classes of (strongly)copure injective and (strongly)copure flat $R$-modules and defined two new homological dimensions related to these classes, namely copure injective and copure flat dimensions. They obtained a characterization for a finite dimensional Gorenstein ring in terms of finiteness of these homological dimensions. Also in \cite{DFZ}, N.Ding et al. defined the dual notion of (strongly)copure injective modules, namely (strongly)copure projective modules and defined the copure projective dimension. In this article we introduce and investigate various classes of $R$-modules that arise from the vanishing of $ \Ext $ and $ \Tor $ on the famous classes $ \mathcal{P}_C(R) $, $ \mathcal{F}_C(R) $ and $ \mathcal{I}_C(R) $. More precisely, in section 4, we define the classes:
\begin{itemize}
	\item[(i)]{$ \mathcal{M}^C_{\scot} = \{\ M \mid \Ext^{i > 0}_R(\mathcal{F}_C(R),M) = 0 \} $}
	\item[(ii)]{$ \mathcal{M}^C_{\stf} = \{\ M \mid \Tor_{i > 0}^R(\mathcal{F}_C(R),M) = 0 \} $}
	\item[(iii)]{$ \mathcal{M}^C_{\sci} = \{\ M \mid \Ext^{i > 0}_R(\mathcal{I}_C(R),M) = 0 \} $}
	\item[(iv)]{$ \mathcal{M}^C_{\scp} = \{\ M \mid \Ext^{i > 0}_R(M, \mathcal{F}_C(R)) = 0 \} $}
	\item[(v)]{$ \mathcal{M}^C_{\scf} = \{\ M \mid \Tor_{i > 0}^R(\mathcal{I}_C(R),M) = 0 \} $}
\end{itemize}
By our definition, it is clear that $ \mathcal{GP}_C \subseteq \mathcal{M}^C_{\scp} $, $ \mathcal{GI}_C \subseteq \mathcal{M}^C_{\sci} $ and $ \mathcal{GF}_C \subseteq \mathcal{M}^C_{\scf} $. We give several characterizations of the above classes and study the relations between them and between the classical ones. For instance, the following is Theorem 4.16.\\
\textbf{Theorem B.} Let $R$ be a Noetherian ring of finite Krull dimension $d$. The following statements hold true:
\begin{itemize}
	\item[(i)]{ Assume that $ M \in \mathcal{M}^C_{\scf} $. Then $ \mathcal{I}_C \otimes_R M \subseteq \mathcal{M}_{\scot}$.}
	\item[(ii)] {Assume that $ M \in \mathcal{M}^C_{\sci} $. Then $ \Hom_R(\mathcal{I}_C , M) \subseteq \mathcal{M}_{\stf}$.}	
\end{itemize}

One of the most influential results in commutative algebra is result of Auslander,
Buchsbaum and Serre: A local ring $ (R, \fm) $ is regular if and only if $ \pd_R(R / \fm) < \infty $. R. Takahashi and D. White, in \cite[Proposition 5.1]{TW}, generalized the mentioned theorem by replacing the condition $ \pd_R(R / \fm) < \infty $ with
$ \mathcal{P}_C $-$ \pd_R(R / \fm) < \infty $. The Auslander-Buchsbaum-Serre theorem, has several generalizations to complete intersection, Gorenstein and Cohen-Macaulay rings by using the finiteness of appropriate homological dimensions of $ R / \fm $. For example, the following result is first proved for a Cohen-Macaulay ring admitting a dualizing module in \cite[6.2.7]{C}, and for an arbitrary Notherian ring by S. Yassemi and L. Khatami \cite[Theorem 2.7]{KY}: A local ring $ (R, \fm) $ is Gorenstein if and only if G-$ \id_R(R / \fm) < \infty $.
In \cite[Theorem 4.1]{EJ3}, E. Enochs and O. Jenda, by using the strongly copure injective and strongly copure flat modules, proved that the following are equivalent for a non-negative integer $n$:
\begin{itemize}
	\item[(i)] {$R$ is $ n $-Gorenstein.}
	\item[(ii)] {$\cid_R(M) \leq n$ for any $R$-module $M$.}	
	\item[(iii)] {$\cfd_R(M) \leq n$ for any $R$-module $M$.}	
\end{itemize}
In section 4, we define three new homological dimensions, $ C $-copure projective dimension
$ C $-$ \scp $, $ C $-copure injective dimension
$ C $-$ \sci $, and $ C $-copure flat dimension
$ C $-$ \scf $. By these new homological dimensions, we generalize \cite[Theorem 4.1]{EJ3}.The following is a part of Theorem 4.27. \\
\textbf{Theorem C.} The following are equivalent:
	\begin{itemize}
		\item[(i)]{ $\id_R(C) \leq n$.}
		\item[(ii)] {$ C $-$\cid_R(M) \leq n$ for any $R$-module $M$.}
		\item[(iv)] {$ C $-$\cpd_R(M) \leq n$ for any $R$-module $M$.}
		\item[(vi)] {$ C $-$\cfd_R(M) \leq n$ for any $R$-module $M$.}
	\end{itemize}

\section{preliminaries}

In this section, we recall some definitions and facts which are needed throughout this
paper. By an injective cogenerator, we always mean an injective $R$-module $E$ for which $ \Hom_R(M,E) \neq 0 $ whenever $M$ is a nonzero $R$-module. For an $R$-module $M$, the injective hull of $M$, is always denoted by $E(M)$.

\begin{defn}
  \emph{Let $\mathcal{X} $ be a class of $R$-modules and $M$ an $R$-module. An $\mathcal{X}$-\textit{resolution} of $M$ is a complex of $R$-modules in $\mathcal{X} $ of the form \\
     \centerline{ $X = \ldots \longrightarrow X_n \overset{\partial_n^X} \longrightarrow X_{n-1} \longrightarrow \ldots \longrightarrow X_1 \overset{\partial_1^X}\longrightarrow X_0 \longrightarrow 0$}
such that $\H_0(X) \cong M$ and $\H_n(X) = 0$ for all $ n \geq 1$.}
\emph{Also the $ \mathcal{X}$-\textit{projective dimension} of $M$ is the quantity \\
\centerline{ $ \mathcal{X}$-$\pd_R(M) := \inf \{ \sup \{ n \geq 0 | X_n \neq 0 \} \mid X$ is an $\mathcal{X}$-resolution of $M \}$}.}
\emph{So that in particular $\mathcal{X}$-$\pd_R(0)= - \infty $. The modules of $\mathcal{X}$-projective dimension zero are precisely the non-zero modules in $\mathcal{X}$. The terms of $\mathcal{X}$-\textit{coresolution} and $\mathcal{X}$-$\id$ are defined dually.}
\end{defn}

\begin{defn}
 \emph{A finitely generated $ R $-module $ C $ is \textit{semidualizing} if it satisfies the following conditions:
\begin{itemize}
             \item[(i)]{The natural homothety map $ R\longrightarrow \Hom_R(C,C) $ is an isomorphism.}
             \item[(ii)]{$ \Ext^i_R(C,C)=0 $ for all $ i>0 $.}
          \end{itemize}}
\end{defn}
For example a finitely generated projective $R$-module of rank 1 is semidualizing. If $R$ is Cohen-Macaulay, then an $R$-module $D$ is dualizing if it is semidualizing and that $\id_R (D) < \infty $ . For example the canonical module of a Cohen-Macaulay local ring, if exists, is dualizing.

\begin{defn}
\emph{Following \cite{HJ}, let $C$ be a semidualizing $R$-module. We set
	\begin{itemize}
		\item[]{$\mathcal{P}_C(R) =$ the subcategory of $R$--modules  $C \otimes_R P$ where $P$ is a projective $R$--module.}
		\item[]{$\mathcal{F}_C(R) =$ the subcategory of $R$--modules  $C \otimes_R F$ where $F$ is a flat $R$--module.}
		\item[]{ $\mathcal{I}_C(R) =$ the subcategory of $R$--modules  $\Hom_R(C,I) $ where $I$ is an injective $R$--module.}
	\end{itemize}
The $R$-modules in $\mathcal{P}_C(R)$, $\mathcal{F}_C(R)$ and $\mathcal{I}_C(R)$ are called $C$-projective, $C$-flat and $C$-injective, respectively.  If $ C = R $, then it recovers the classes of projective modules $ \mathcal{P} $, flat modules $ \mathcal{F} $ and injective modules $  \mathcal{I}  $, respectively.
We use the notations $C$-$\pd$, $C$-$\fd$ and $C$-$\id$ instead of  $\mathcal{P}_C$-$\pd$, $\mathcal{F}_C$-$\pd$ and $\mathcal{I}_C$-$\id$, respectively.}
\end{defn}

Based on the work of E. Enochs and O. Jenda \cite{EJ1}, the following notions were
introduced and studied by H. Holm and P. Jørgensen \cite{HJ}.

\begin{defn}
	 \emph{A complete $\mathcal{P} \mathcal{P}_C$-\textit{resolution} is a complex $X$ of $R$-modules such that
	 	\begin{itemize}
	 		\item[(i)]{$X$ is exact and $\Hom_R(X,P)$ is exact for each $P \in \mathcal{P}_C(R)$, and that}
	 		\item[(ii)]{$X_i \in \mathcal{P}_C(R)$ for all $i < 0$ and $X_i$ is projective for all $i \geq 0$.}
	 	\end{itemize}
	 	An $R$-module $M$ is called G$_C$-\textit{projective} if there exists a complete $\mathcal{P} \mathcal{P}_C$-resolution $X$ such that
	 	$M \cong \coker (\partial_1^X)$. All projective $R$-modules and all $R$-modules in $\mathcal{P}_C(R)$ are G$_C$-projective. The class of G$_C$-projective $R$-modules is denoted by $\mathcal{GP}_C$. Also we use the notation G$_C$-$\pd$ instesd of $\mathcal{GP}_C$-$\pd$. }
	
 \emph{A complete $\mathcal{F} \mathcal{F}_C$-{resolution} is a complex $X$ of $R$-modules such that
\begin{itemize}
	\item[(i)]{$X$ is exact and $X \otimes_R I$ is exact for each $I \in \mathcal{I}_C(R)$, and that}
	\item[(ii)]{$X_i \in \mathcal{F}_C(R)$ for all $i < 0$ and $X_i$ is flat for all $i \geq 0$.}
\end{itemize}
An $R$-module $M$ is called G$_C$-\textit{flat} if there exists a complete $\mathcal{F} \mathcal{F}_C$-resolution $X$ such that
$M \cong \coker (\partial_1^X)$. All flat $R$-modules and all $R$-modules in $\mathcal{F}_C(R)$ are G$_C$-flat. The class of G$_C$-flat $R$-modules is denoted by $\mathcal{GF}_C$. Also we use the notation G$_C$-$\fd$ instesd of $\mathcal{GF}_C$-$\pd$. }

\emph{A complete $\mathcal{I}_C \mathcal{I}$-{coresolution} is a complex $Y$ of $R$-modules such that
\begin{itemize}
	\item[(i)]{$Y$ is exact and $\Hom_R(I,Y)$ is exact for each $I \in \mathcal{I}_C(R)$, and that}
	\item[(ii)]{$Y^i \in \mathcal{I}_C(R)$ for all $i \geq 0$ and $Y^i$ is injective for all $i < 0$.}
\end{itemize}
An $R$-module $M$ is called G$_C$-\textit{injective} if there exists a complete $\mathcal{I}_C \mathcal{I}$-coresolution $Y$ such that
$M \cong \coker (\partial_1^Y)$. All injective $R$-modules and all $R$-modules in $\mathcal{I}_C(R)$ are G$_C$-injective. The class of G$_C$-injective $R$-modules is denoted by $\mathcal{GI}_C$. Also we use the notation G$_C$-$\id$ instesd of $\mathcal{GI}_C$-$\id$. \\
Note that when $C = R$ these notions recover the concepts of Gorenstein projective module $ \mathcal{GP} $, Gorenstein flat module $ \mathcal{GF} $
and Gorenstein injective modules $ \mathcal{GI} $ which were introduced in \cite{EJT} and \cite{EJ2}.}
\end{defn}
\begin{lem}
	Let $C$ be a semidualizing $R$-module and let $M$ be an $R$-module.
	
	\begin{itemize}
		\item[(i)]{$C$-$\emph{\id}_R(M) = \emph{\id}_R (C \otimes_R M) $ and $\emph{\id}_R(M) =C$-$\emph{\id}_R(\emph{\Hom}_R(C,M))$}.
		\item[(ii)]{$C$-$\emph{\pd}_R(M) = \emph{\pd}_R (\emph{\Hom}_R(C,M))$ and $\emph{\pd}_R(M) =C$-$\emph{\pd}_R(C \otimes_R M) $}.
		\item[(iii)]{$C$-$\emph{\fd}_R(M) = \emph{\fd}_R (\emph{\Hom}_R(C,M))$ and $\emph{\fd}_R(M) =C$-$\emph{\fd}_R(C \otimes_R M) $}.
	\end{itemize}
\end{lem}
\begin{proof}
	For (i) and (ii), see \cite[Theorem 2.11]{TW} and for (iii), see \cite[Proposition 5.2]{STWY}.
\end{proof}
\begin{lem}
Let $M$ be an $R$-module and let $E$ be an injective cogenerator. Then $M \in \mathcal{GF}_C$ if and only
 if $ \emph{Hom}_R(M,E) \in \mathcal{GI}_C$.	
\end{lem}
\begin{proof}
Let $ R \ltimes C $ denote the trivial extension of $R$ by $C$ and view M as an
$ R \ltimes C $-module via the natural surjection $ R \ltimes C \rightarrow R $. Now $M \in \mathcal{GF}_C$ if and only if $M$ is G-flat
 over $ R \ltimes C $ by \cite[Theorem 2.16]{HJ}. This is the case if and only if $ \Hom_R(M,E) $ is G-injective over $ R \ltimes C $ and using \cite[Theorem 2.16]{HJ} once more completes the proof.
\end{proof}	
	
\begin{defn}
\emph{Let $C$ be a semidualizing $R$-module. The \textit{Auslander class with
respect to} $C$ is the class $\mathcal{A}_C(R)$ of $R$-modules $M$ such that:
\begin{itemize}
	\item[(i)]{$\Tor_i^R(C,M) = 0 = \Ext^i_R(C, C \otimes_R M)$ for all $i \geq 1$, and}
	\item[(ii)]{The natural map $ M \rightarrow \Hom_R(C , C \otimes_R M )$ is an isomorphism.}
\end{itemize}		
The \textit{Bass class with
respect to} $C$ is the class $\mathcal{B}_C(R)$ of $R$-modules $M$ such that:
\begin{itemize}
	\item[(i)]{$\Ext^i_R(C,M) = 0 = \Tor_i^R(C, \Hom_R(C,M))$ for all $i \geq 1$, and}
	\item[(ii)]{The natural map $ C \otimes_R \Hom_R(C,M)) \rightarrow M $ is an isomorphism.}
\end{itemize}
The class $\mathcal{A}_C(R)$ contains all $R$-modules of finite projective dimension and those of finite $C$-injective dimension. Also the class $\mathcal{B}_C(R)$ contains all $R$-modules of finite injective dimension and those of finite $C$-projective dimension (see \cite[Corollary 2.9]{TW}). Also, if any two $ R $-modules in a short exact sequence are in $ \mathcal{A}_C(R) $ (resp. $ \mathcal{B}_C(R) $), then so is the third (see \cite{HW}).}
\end{defn}

\begin{defn}
\emph{ Let $ \mathcal{X} $ be a class of $R$-modules. Following \cite{X}, we say that $ \mathcal{X} $ is \textit{closed under extension} whenever
 $ 0 \rightarrow X' \rightarrow X \rightarrow X'' \rightarrow 0 $ is any exact sequence, with $ X' , X'' \in \mathcal{X} $, then $ X \in \mathcal{X} $. Also the class $ \mathcal{X} $ is said to be \textit{resolving} (resp. \textit{coresolving}), provided that $ \mathcal{X} $ is closed under extensions, $  \mathcal{P} \subseteq \mathcal{X} $ (resp. $  \mathcal{I} \subseteq \mathcal{X} $) and
 $ X' \in  \mathcal{X}  $ (resp. $ X'' \in  \mathcal{X}  $) whenever $ 0 \rightarrow X' \rightarrow X \rightarrow X'' \rightarrow 0 $ a short exact sequence such that $ X , X'' \in \mathcal{X} $ (resp $ X , X' \in \mathcal{X} $).	}
\end{defn}	
\begin{defn}
\emph{Let $M$ be an $R$-module and let $ \mathcal{X} $ be a class of $R$-modules . Following \cite{EJ1}, a $ \mathcal{X} $-\textit{precover} of $M$ is a homomorphism
$ \varphi : X \rightarrow M$, with $X \in \mathcal{X}$, such that every homomorphism $Y \rightarrow M$ with $Y \in \mathcal{X}$, factors through $\phi$;
i.e., the homomorphism \\
\centerline{$\Hom_R(Y, \varphi): \Hom_R(Y,X) \rightarrow \Hom_R(Y,M)$}
is surjective for each module $Y$ in $ \mathcal{X} $. A $ \mathcal{X} $-precover $ \varphi : X \rightarrow M$ is a $ \mathcal{X} $-\textit{cover} if every $ \psi \in \Hom_R(X,X)$ with $\varphi \psi = \varphi$ is an automorphism. \\
Dually, a $ \mathcal{X} $-\textit{preenvelope} of $M$ is a homomorphism
$ \varphi : M \rightarrow X$, with $X \in \mathcal{X}$, such that every homomorphism $M \rightarrow Y$ with $Y \in \mathcal{X}$, factors through $\phi$; i.e., the homomorphism \\
\centerline{$\Hom_R(\varphi , Y ): \Hom_R(X,Y) \rightarrow \Hom_R(M,Y)$}
is surjective for each module $Y$ in $ \mathcal{X} $. A $ \mathcal{X} $-preenvelpoe $ \varphi : M \rightarrow X$
is a $ \mathcal{X} $-\textit{envelope} if every $ \psi \in \Hom_R(X,X)$ with $ \psi \varphi = \varphi$ is an automorphism. \\
We say that a class $ \mathcal{X} $ is \textit{$ ( $pre$ ) $covering} in the category of $R$-modules, precisely when any $R$-module has a $ \mathcal{X} $-(pre)cover. Dually, a class $ \mathcal{X} $ is \textit{$ ( $pre$ ) $enveloping }in the category of $R$-modules, precisely when any $R$-module has a $ \mathcal{X} $-(pre)envelope.}	
\end{defn}

\begin{lem}\label{A2}
Let $C$ be a semidualizing $R$-module. We have the following statements:
\begin{itemize}
	\item[(i)]{The class $ \mathcal{I}_C(R)$ is covering and enveloping on the category of $R$-modules and the class $ \mathcal{P}_C(R)$ is precovering. Also, the class $ \mathcal{F}_C(R)$ is covering and preenveloping on the category of $R$-modules.}
	\item[(ii)]{Assume that $M$ is an $R$-module. If $M \in \mathcal{A}_C(R)$ $ ( $resp. $M \in \mathcal{B}_C(R)$ $ ) $, then any $ \mathcal{I}_C$-coresolution $ ( $resp. $\mathcal{P}_C$-resolution$ ) $ of $M$ is exact and $ (C \otimes_R -) \emph{-}$exact $ ( $resp. $ \emph{\Hom}_R(C,-) \emph{-}$exact$ ) $. In particular, any $ \mathcal{I}_C$-preenvelope $ ( $resp. $\mathcal{P}_C$-precover$ ) $ of $M$ is injective $ ( $resp. surjective$ ) $ .}
\end{itemize}
\end{lem}
\begin{proof}
	For (i), see \cite[Proposition 5.3]{HW} and for (ii), see \cite[Corollary 2.4]{TW}.
\end{proof}

\begin{defn}
\emph{Let $ \mathcal{X} $ be class of $R$-modules. Set $ \mathcal{X}^{\perp} = \{\ N \mid \Ext^1_R(\mathcal{X} , N) = 0 \} $ and
$ ^{\perp}\mathcal{X} = \{\ M \mid \Ext^1_R(M , \mathcal{X}) = 0\}$. Following \cite{EJ1}, the pair $ (\mathcal{X} , \mathcal{Y}) $ is called a \textit{cotorsion theory}, precisely when $ \mathcal{X} = ^{\perp}\mathcal{Y} $ and $ \mathcal{X}^{\perp} = \mathcal{Y}$. Also a cotorsion theory $ (\mathcal{X} , \mathcal{Y}) $ is said to be \textit{hereditary}, precisely when $ \Ext^i_R(\mathcal{X} , \mathcal{Y}) = 0 $ for all $ i \geq 1 $.}
\end{defn}
	
\begin{defn}
\emph{Following \cite{E2}, an $R$-module $M$ is called  \textit{cotorsion} if $ \Ext^1_R(F,M) = 0 $ for any flat $R$-module F. Following \cite{X}, an $R$-module $M$ is
	 called  \textit{strongly cotorsion }
if $ \Ext^1_R(F,M) = 0 $ for any $R$-module $F$ with finite flat dimension. An $R$-module N is called  \textit{strongly torsionfree}  if
$\Tor^R_1(N,X) = 0$ for any $R$-module X with finite flat dimension. Following \cite{EJ3}, an $R$-module $X$ is said to be \textit{strongly copure injective}
if $ \Ext^1_R(E,X) = 0 $ for any $R$-module $E$ with finite injective dimension. An $R$-module $Y$ is said to be \textit{strongly copure  flat}
if $ \Tor^R_1(E,Y) = 0 $ for any $R$-module $E$ with finite injective dimension. Following \cite{DFZ}, an $R$-module $Z$ is said to be \textit{strongly copure projective}
if $ \Ext^1_R(Z,F) = 0 $ for any flat $R$-module $F$.}	
\end{defn}
Note that if $M$ is strongly cotorsion (resp. strongly torsionfree) then $ \Ext^i_R(F,M) = 0 $ (resp. $ \Tor^R_i(F,M) = 0 $) for any $R$-module $F$ with finite flat dimension. Also note that any G-injective (resp. G-flat) $R$-module is strongly copure injective (resp. strongly copure flat).
\begin{rem}
\emph{In view of Definition 2.12, for any $R$-module $M$, we can define two homological dimensions, namely strongly cotorsion dimension $ \scd_R(M) $, and strongly torsionfree dimension $ \std_R(M) $. Then it is not hard to see that there are equalities $ \scd_R(M) = \sup \{ n \geq 0 | \Ext^n_R(F , M) \neq 0 $, where $F$ is an $R$-module with $\fd_R(F) < \infty $ $\}$, $ \std_R(M) = \sup \{ n \geq 0 | \Tor^R_n(F , M) \neq 0 $, where $F$ is an $R$-module with $\fd_R(E) < \infty $ $\}$. In \cite{DFZ}, E. Enochs and O. Jenda, and in \cite{EJ3}, N. Ding et al. defined three new homological dimensions for modules, namely copure injective dimension denoted by $\cid$, copure flat dimension denoted by $\cfd$ and  copure projective dimension denoted by $\cpd$. It is not hard to see that for an $R$-module $M$, there are equalities $ \cid_R(M) = \sup \{ n \geq 0 | \Ext^n_R(E , M) \neq 0 $, where $E$ is an $R$-module with $\id_R(E) < \infty $ $\}$, $ \cfd_R(M) = \sup \{ n \geq 0 | \Tor^R_n(E , M) \neq 0 $, where $E$ is an $R$-module with $\id_R(E) < \infty $ $\}$ and $ \cpd_R(M) = \sup \{ n \geq 0 | \Ext^n_R(M , F) \neq 0 $, where $F$ is an $R$-module with $\fd_R(F) < \infty $ $\}$.}
\end{rem}
\begin{no}
\emph{We use the notations $ \mathcal{M}_{\Cot} $, $ \mathcal{M}_{\scot} $, $ \mathcal{M}_{\stf} $, $ \mathcal{M}_{\ci} $, $ \mathcal{M}_{\sci} $, $ \mathcal{M}_{\scf} $ and $ \mathcal{M}_{\scp} $ to denote the full subcategories of cotorsion, strongly cotorsion, strongly torsionfree, copure injective, strongly copure injective, strongly copure flat and strongly copure projective $R$-modules, respectively. It is clear that $ \mathcal{GP} \subseteq \mathcal{M}_{\scp} $, $ \mathcal{GF} \subseteq \mathcal{M}_{\scf} $	and that $ \mathcal{GI} \subseteq \mathcal{M}_{\sci} $.}
\end{no}
\begin{lem}\label{A2}
Let $M$ be an $R$-module and let $E$ be an injective cogenerator. The following statements hold:
\begin{itemize}
	\item[(i)]{One has $ M \in \mathcal{M}_{\emph{\stf}} $ if and only if $ \emph{Hom}_R(M,E) \in \mathcal{M}_{\emph{\scot}} $.}
	\item[(ii)]{One has $ M \in \mathcal{M}_{\emph{\scf}} $ if and only if $ \emph{Hom}_R(M,E) \in \mathcal{M}_{\emph{\sci}} $.}
\end{itemize}
\end{lem}	
\begin{proof}
(i). Assume that $F$ is an $R$-module of finite flat dimension. By \cite[theorem 3.2.1]{EJ1}, there is an isomorphism \\
\centerline{$ \Hom_R(\Tor^R_i(F,M), E) \cong \Ext^i_R(F, \Hom_R(M,E))$,}
For all $ i \geq 1 $. Hence $ \Tor^R_1(F,M) = 0 $ if and only if $\Ext^1_R(F, \Hom_R(M,E)) = 0$. The proof of (ii) is similar.
\end{proof}		

\section{G$ _C $-injective, G$ _C $-flat and Dualizing Modules}
Throughout $C$ is a semidualizing $R$-module. Our aim in this section is to show that if $C$ is dualizing, then there are equalities of subcategories $ \mathcal{GI}_C = \mathcal{M}_{\scot} $ and $ \mathcal{GF}_C = \mathcal{M}_{\stf} $. We begin with a proposition which says that the containment '$ \subseteq $' does not need the dualizing assumption on $C$.
\begin{prop}\label{A2}
Let $R$ be a ring of finite Krull dimension $d$. Then we have the following statements:
\begin{itemize}
	\item[(i)]{One has $ \mathcal{GI}_C \subseteq \mathcal{M}_{\emph{\scot}} $.}
	\item[(ii)]{One has $ \mathcal{GF}_C \subseteq \mathcal{M}_{\emph{\stf}} $.}
\end{itemize}
\end{prop}
\begin{proof}
(i). First we show that $ \mathcal{I}_C \subseteq \mathcal{M}_{\scot} $. Let $F$ be an $R$-module of finite flat dimension. By \cite[theorem 3.2.1]{EJ1}, the isomorphism \\
\centerline{ $\Ext^1_R(F , \Hom_R(C,I)) \cong \Hom_R(\Tor^R_1(C,F) , I)$,}
holds for any injective $R$-module $I$. Now since $F \in \mathcal{A}_C(R)$ we have $\Tor^R_i(C,F) = 0$ for all $i \geq 1$. In particular,
we have $\Tor^R_1(C,F) = 0$, which implies that $\Ext^1_R(F , \Hom_R(C,I)) = 0$, as wanted. Now let $M \in \mathcal{GI}_C$. In view of \cite[Lemma 2.4]{HJ}, we have exact sequences \\
\centerline{ $0 \rightarrow M_{i+1} \rightarrow E_i \rightarrow M_i \rightarrow 0 $,}
in which $M_i \in \mathcal{GI}_C$ and $E_i \in \mathcal{I}_C$ for all $i \geq 0$ such that $M_0 = M$. Application of the
functor $\Hom_R(F,-)$ on these exact sequences, yields the isomorphisms \\
\centerline{ $\Ext^1_R(F , M_0) \cong \Ext^2_R(F , M_1) \cong \cdots \cong \Ext^{d + 1}_R(F , M_d) $.}
According to \cite[Corollary 3.4]{F1}, we have $\pd_R (F) \leq d$. Thus, using the above isomorphisms, we have $\Ext^1_R(F , M_0) \cong \Ext^{d + 1}_R(F , M_{d}) = 0 $, which completes the proof.

(ii). Let $N \in \mathcal{GF}_C$ and let $E$ be an injective cogenerator. Then by Lemma 2.6 $\Hom_R(N,E) \in \mathcal{GI}_C$ and hence $ \Hom_R(N,E) \in \mathcal{M}_{\scot} $ by (i). Therefore $N \in \mathcal{M}_{\stf}$ by Lemma 2.15.
\end{proof}

\begin{lem}\label{A2}
Let R be a ring of finite Krull dimension d and let $C$ be dulalizing.
 \begin{itemize}
	\item[(i)]{An $R$-module $M$ is \emph{G}$_C$\emph{-}injective if and only if there exists an exact
sequence \\
$ \cdots \rightarrow \emph{\Hom}_R(C,I_n) \rightarrow \cdots \rightarrow \emph{\Hom}_R(C,I_1) \rightarrow \emph{\Hom}_R(C,I_0) \rightarrow M \rightarrow 0$, \\
in which $I_i$ is injective for all $i \geq 0$.}
	\item[(ii)]{An $R$-module $N$ is \emph{G}$_C$\emph{-}flat if and only if there exists an exact
sequence \\
\centerline{$ 0 \rightarrow N \rightarrow C \otimes_R F^0 \rightarrow C \otimes_R F^1 \rightarrow \cdots \rightarrow C \otimes_R F^n \rightarrow \cdots $,}
 in which $F^i$ is flat for all $i \geq 0$.}
\end{itemize}
\end{lem}
\begin{proof}
(i). The definition of a G$_C$-injective $R$-module implies the existence of such an exect sequence.  Now assume that $M$ is an $R$-module for which there exists an exact sequence\\
\centerline{$ \cdots \rightarrow \Hom_R(C,I_n) \rightarrow \cdots \rightarrow \Hom_R(C,I_1) \rightarrow \Hom_R(C,I_0) \rightarrow M \rightarrow 0$, $(*)$}
with $I_i$ is injective for all $i \geq 0$. By the definition, we have to show that the sequence $(*)$ is $\Hom_R(\mathcal{I}_C,-)$-exact and that $ \Ext^i_R(\mathcal{I}_C , M) = 0 $ for all $i \geq 0$. Note that for any injective $R$-module $I$, we have $\fd_R(\Hom_R(C,I)) \leq d$ since $ \id_R(C) =d $. Hence \cite[Corollary 3.4]{F1} implies that $\pd_R (\Hom_R(C,I)) \leq d$. Therefore
all modules in $ \mathcal{I}_C $ have finite projective dimension less than or equal to $d$. Now breaking up $(*)$ to short exact sequences and using the same argument as in proof of Proposition 3.1, one can see that $ \Ext^i_R(\mathcal{I}_C , M) = 0 $ for
 all $i \geq 0$. Next suppose that $K_j$ is the image of the $j$-th boundary map of the complex $(*)$ for $j \geq 1$ . Then $K_j$ has a left resolution by modules in $ \mathcal{I}_C $. hence the same argument shows that $ \Ext^1_R(\mathcal{I}_C , K_j) = 0 $. Consequently, the exact sequence $(*)$ is $\Hom_R(\mathcal{I}_C,-)$-exact, whence $M$ is G$_C$-injective.

 (ii). If $N$ is G$_C$-flat then there exists such an exact sequence. Now assume that $N$ is an $R$-module for which there exists an exaxt sequence\\
 \centerline{$ 0 \rightarrow N \rightarrow C \otimes_R F^0 \rightarrow C \otimes_R F^1 \rightarrow \cdots \rightarrow C \otimes_R F^n \rightarrow \cdots $, $ (**) $}
in which $F^i$ is flat for all $i \geq 0$. Suppose that $E$ is an injective cogenerator. Then by applying the exact functor $\Hom_R( - , E)$ to $(**)$, we obtain an exact sequence\\
\centerline{$ \cdots \rightarrow \Hom_R(C \otimes_R F^n , E) \rightarrow \cdots \rightarrow \Hom_R(C \otimes_R F^n ,E) $}
\centerline{$ \rightarrow \Hom_R(C \otimes_R F^n , E) \rightarrow \Hom_R( N , E) \rightarrow 0 $.}
Note that $\Hom_R(C \otimes_R F^i , E)\cong \Hom_R(C  , \Hom_R( F^i , E))$ and that $\Hom_R( F^i , E)$ is an injective $R$-module for all $i \geq 0$. Thus (ii) follows from (i) and Lemma 2.6.
\end{proof}
\begin{thm}\label{A2}
Let R be a ring of finite Krull dimension d and let $C$ be dulalizing. Then any strongly cotorsion $R$-module is $\emph{G}_C\emph{-}$injective.
\end{thm}
\begin{proof}
Let $M$ be a strongly cotorsion $R$-module. In view of lemma 3.2, we need only to construct an exact sequence of the form \\
 \centerline{$ \cdots \rightarrow \Hom_R(C,I_n) \rightarrow \cdots \rightarrow \Hom_R(C,I_1) \rightarrow \Hom_R(C,I_0) \rightarrow M \rightarrow 0$,}
 in which $I_i$ is injective for all $i \geq 0$. Assume that $ \varphi : \Hom(C,I_0) \rightarrow M$ is the $\mathcal{I}_C$-precover of $M$. First we show that $ \varphi $ is surjective and that $ \ker \varphi $ is cotorsion. We can consider a surjective homomorphism
 $ \theta : P \rightarrow M $ with $P$ is projective. Now the $\mathcal{I}_C$-preenvelope of $P$ is injective by Lemma 2.10(ii), since
 $ P \in \mathcal{A}_C(R) $. Assume that $ \tau : P \rightarrow \Hom(C,E) $ is an injective $\mathcal{I}_C$-preenvelope of $P$. Consider the following push-out diagram:
 \begin{displaymath}
 \xymatrix{
 	&&  0  \ar[d] &  0  \ar[d]  \\
 	  0 \ar[r] & K \ar[r] \ar@{=}[d] & P \ar[r]^{\varphi} \ar[d]_{\tau} & M \ar[r] \ar[d] & 0 \\
    0 \ar[r] & K \ar[r] & \Hom_R(C,E) \ar[r] \ar[d] & X \ar[r] \ar[d] & 0 \\
    &&  Y \ar@{=}[r] \ar[d] & Y \ar[d] \\
    &&  0 &  0             \\
 }
 \end{displaymath}
 where $K = \ker \varphi $ and $Y= \coker \tau$. Note that $\fd_R (\Hom_R(C,E)) < \infty$ and that $P$ is flat. Hence $\fd_R(Y) < \infty$ since the middle column is exact. Now since $M$ is strongly cotorsion, we have $ \Ext^1_R(Y,M) = 0 $ which shows that the right hand side column is split. Consequently, there is an epimorphism $ \Hom_R(C , E) \rightarrow M $ from which we conclude that the $\mathcal{I}_C$-precover of $M$ is surjective. Therefore we have the exact sequence \\
\centerline{$ 0 \rightarrow Z \rightarrow \Hom(C,I_0) \overset{\varphi}\rightarrow M \rightarrow 0 ,$ $ (*) $}
in which $ Z = \ker \varphi$. Next we show that $Z$ is strongly cotorsion. Assume that $F$ is an $R$-module of finite flat dimension. Application of the functor $\Hom_R(F , -)$ on the exact sequence $ (*) $ yields  $ \Ext^2_R(F , Z) \cong \Ext^1_R(F , M) = 0 $. The fact that
$ F \in \mathcal{A}_C(R) $ implies that the  $\mathcal{I}_C$-preenvelope of $F$ is injective by Lemma 2.10(ii). Hence there is an exact sequence \\
\centerline{$ 0 \rightarrow F \rightarrow \Hom(C,E') \rightarrow W \rightarrow 0 ,$ $ (**) $}
in which $E'$ is injective and $ W = \coker (F \rightarrow \Hom_R(C,E'))$. Again $ \fd_R(W) < \infty $ since both $F$ and $ \Hom_R(C,E') $ have finite flat dimensions. But then $ \Ext^2_R(W , Z) = 0$, as we have seen before. An application of the functor $ \Hom_R(- , Z ) $ on the exact sequence $ (**) $ yields $ \Ext^1_R(F , Z) \cong \Ext^2_R(W , Z) = 0 $, which shows that $Z$ is cotorsion. Finally, we can proceed in the same way to get the exact sequence \\
 \centerline{$ \cdots \rightarrow \Hom_R(C,I_n) \rightarrow \cdots \rightarrow \Hom_R(C,I_1) \rightarrow \Hom_R(C,I_0) \rightarrow M \rightarrow 0$.}
 Hence $M$ is G$_C$-injective by Lemma 3.2.
\end{proof}
\begin{cor}\label{A2}
Let R be a ring of finite Krull dimension d and let $C$ be dulalizing.
\begin{itemize}
	\item[(i)]{One has $ \mathcal{GI}_C = \mathcal{M}_{\emph{\scot}} $.}
	\item[(ii)] {One has $ \mathcal{GF}_C = \mathcal{M}_{\emph{\stf}} $.}
 \end{itemize}
  \end{cor}
 \begin{proof}
 	For (i), use Proposition 3.1(i) and Theorem 3.3. For (ii), use part (i), Lemma 2.6, Lemma 2.15 and Proposition 3.1(i).
 \end{proof}

\section{Generalizations}
In the present section, by using a semidualizing module, we extend the notions of Definition 2.12 and the homological dimensions described in the Remark 2.13. First, we define some new classes of modules. Next, we study the properties of these classes and related homological dimensions. One of our main goals is to impose suitable conditions on $C$ to be dualizing. The next goal is to find relations between these new classes with those of Notion 2.14. To do this, we use some properties of Auslander and Bass classes.  
\begin{defn}
\emph{A submodule $ N $ of an $R$-module $ M $ is said to be a $ C $-\textit{copure} submodule if $ M/N \in \mathcal{I}_C$, and
an $R$-module $M$ is said to be $ C $-\textit{copure injective} (resp. $ C $-\textit{copure flat}) if it is injective (resp. flat) with respect to all $ C $-copure
submodules; i.e.,  $ M $ is $ C $-copure injective (resp. $ C $-copure flat) if $ \Ext^1_R(\mathcal{I}_C , M) = 0 $
(resp. $ \Tor^R_1(\mathcal{I}_C , M) = 0 $). An $R$-module $ M $ is said to be
a \textit{strongly} $ C $-\textit{copure injective} (resp. \textit{strongly} $ C $-\textit{copure flat}) $R$-module if $ \Ext^i_R(\mathcal{I}_C , M) = 0 $
(resp. $ \Tor^R_i(\mathcal{I}_C , M) = 0 $) for all $ i \geq 1 $. An $R$-module $M$  is said to be \textit{strongly} $ C $-\textit{copure projective} if
$ \Ext^i_R(M , \mathcal{F}_C) = 0 $ for all $ i \geq 1 $.
An $R$-module $M$ is said to be $C$-\textit{cotorsion} if $ \Ext^1_R(\mathcal{F}_C , M) = 0 $. Also $ M $ is said to be
a \textit{strongly} $ C $-\textit{cotorsion} (resp. \textit{strongly} $ C $-\textit{torsionfree}) $R$-module if $ \Ext^i_R(\mathcal{F}_C , M) = 0 $ (resp. $ \Tor^R_i(\mathcal{F}_C , M) = 0 $) for all $ i \geq 1 $.}
\end{defn}

\begin{no}
\emph{We use the notations $ \mathcal{M}^C_{\Cot} $, $ \mathcal{M}^C_{\ci} $, $ \mathcal{M}^C_{\cf} $, and $ \mathcal{M}^C_{\cp} $ to denote the full subcategories of $ C $-cotorsion, $ C $-copure injective,  $ C $-copure flat and $ C $-copure projective modules. Also, we use the notations $ \mathcal{M}^C_{\scot} $, $ \mathcal{M}^C_{\stf} $, $ \mathcal{M}^C_{\sci} $, $ \mathcal{M}^C_{\scf} $ and $ \mathcal{M}^C_{\scp} $ to denote the full subcategories of strongly $ C $-cotorsion, strongly $ C $-torsionfree, strongly $ C $-copure injective, strongly $ C $-copure flat $R$-modules and strongly $ C $-copure projective $R$-modules, respectively.
By the above definition, it is clear that $ \mathcal{GP}_C \subseteq \mathcal{M}^C_{\scp} $, $ \mathcal{GI}_C \subseteq \mathcal{M}^C_{\sci} $ and $ \mathcal{GF}_C \subseteq \mathcal{M}^C_{\scf} $. }
\end{no}
It is clear that the classes $ \mathcal{M}^C_{\scp} $, $ \mathcal{M}^C_{\scf} $ and $ \mathcal{M}^C_{\stf} $ are resolving and that the classes $ \mathcal{M}^C_{\sci} $ and $ \mathcal{M}^C_{\scot} $ are coresolving. In particular, the classes $ \mathcal{M}_{\scp} $, $ \mathcal{M}_{\scf} $ and $ \mathcal{M}_{\stf} $ are resolving and that the classes $ \mathcal{M}_{\sci} $ and
$ \mathcal{M}_{\scot} $ are coresolving.

\begin{prop}
The following statements hold true:
\begin{itemize}
	\item[(i)]{The class $ \mathcal{M}^C_{\emph{\sci}} $  is closed under arbitrary direct product.}
	\item[(ii)]{The classes $ \mathcal{M}^C_{\emph{\scp}} $ and $ \mathcal{M}^C_{\emph{\scf}} $ are closed under arbitrary direct sum.}
\end{itemize}
\end{prop}
\begin{proof}
	(i). Suppose that $ \{M_{\alpha}\}_{\alpha \in I} $ is a collection of modules in $ \mathcal{M}^C_{\sci} $. By \cite[Proposition 7.22]{R}, there is an isomorphism $ \Ext_R^i(\mathcal{I}_C,  \underset{\alpha \in I} \prod M_{\alpha} ) \cong \underset{\alpha \in I} \prod \Ext_R^i(\mathcal{I}_C, M_{\alpha} ) $ for any $ i \geq 0 $.
	
	(i). Suppose that $ \{M_{\alpha}\}_{\alpha \in I} $ and $ \{N_{\beta}\}_{\beta \in J} $ are collections of modules in $ \mathcal{M}^C_{\scp} $ and $ \mathcal{M}^C_{\scf} $, respectively. By \cite[Proposition 7.21 and 7.6]{R}, there are isomorphisms $ \Ext_R^i(\underset{\alpha \in I} \bigoplus M_{\alpha} , \mathcal{F}_C ) \cong \underset{\alpha \in I} \prod \Ext_R^i(M_{\alpha},\mathcal{F}_C) $ and $ \Tor_i^R(\mathcal{I}_C,  \underset{\beta \in J} \bigoplus N_{\beta} ) \cong \underset{\beta \in J} \bigoplus \Tor_i^R(\mathcal{I}_C, N_{\beta} ) $ for any $ i \geq 0 $.	
\end{proof}
\begin{prop}
The following statements hold true:
	\begin{itemize}
		\item[(i)]{The pair $ \Big( $ $ ^{\perp}\mathcal{M}^C_{\emph{\sci}} , \mathcal{M}^C_{\emph{\sci}} $ $ \Big) $ is a hereditary cotorsion theory.}
		\item[(ii)]{The pair $ \Big( $ $ ^{\perp}\mathcal{M}^C_{\emph{\scot}} , \mathcal{M}^C_{\emph{\scot}} $ $ \Big) $ is a hereditary cotorsion theory.}
		\item[(iii)]{The pair $ \Big( $ $ \mathcal{M}^C_{\emph{\scp}} , (\mathcal{M}^C_{\emph{\scp}})^{\perp} $ $ \Big) $ is a hereditary cotorsion theory.}
	\end{itemize}
\end{prop}
\begin{proof}
	(i). We have to show that $ \Big( ^{\perp}\mathcal{M}^C_{\sci} \Big)^{\perp} = \mathcal{M}^C_{\sci} $. The inclusion '$ \supseteq $' is trivial. Suppose that $ M \in \Big( ^{\perp}\mathcal{M}^C_{\sci} \Big)^{\perp}  $. We show that $ \Ext^i_R(^{\perp}\mathcal{M}^C_{\sci} ,M) =0 $ for all $ i \geq 1 $. Choose $ A \in ^{\perp}\mathcal{M}^C_{\sci} $ and $ B \in \mathcal{M}^C_{\sci} $. By our choice, we have
	$ \Ext^1_R(A ,M) = 0 $. There exist two exact sequences $ 0 \rightarrow A_1 \rightarrow P \rightarrow A \rightarrow 0 $ and
	$ 0 \rightarrow B \rightarrow E \rightarrow B^1 \rightarrow 0 $ in which $P$ is projective and $E$ is injective. Since the subcategory
	$ \mathcal{M}^C_{\sci} $ is coresolving, we have $ B^1 \in \mathcal{M}^C_{\sci} $. Hence $ \Ext^2_R(A , B) \cong \Ext^1_R(A , B^1) = 0 $. On the other hand, we have the isomorphism $\Ext^1_R(A_1 , B) \cong \Ext^2_R(A , B) $ which shows that $ \Ext^1_R(A_1 , B) = 0 $, whence
	$ A_1 \in ^{\perp}\mathcal{M}^C_{\sci} $. Therefore $ \Ext^2(A,M) \cong \Ext^1(A_1,M) = 0 $. Now we can proceed in the same way to see that $ \Ext^i(A,M) = 0 $ for all $ i \geq 1 $; i.e. $ \Ext^i_R(^{\perp}\mathcal{M}^C_{\sci} ,M) =0 $ for all $ i \geq 1 $. In particular,
	$ \Ext^i_R(\mathcal{I}_C ,M) =0 $ for all $ i \geq 1 $, whence $ M \in \mathcal{M}^C_{\sci} $. Hence $ \Big( ^{\perp}\mathcal{M}^C_{\sci} , \mathcal{M}^C_{\sci} \Big)$ is a cotorsion theory. Now since $ \mathcal{M}^C_{\sci} $ is coresolving, it is easy to see that $ \Big( ^{\perp}\mathcal{M}^C_{\sci} , \mathcal{M}^C_{\sci} \Big)$ is hereditary.
	
	(ii). Is similar to (i).
	
	(iii). Is dual to (i).
\end{proof}
The following lemma is similar to Lemma 2.15 and we left the proof to the reader.
\begin{lem}\label{A2}
	Let $M$ be an $R$-module and let $E$ be an injective cogenerator. The following statements hold:
	\begin{itemize}
		\item[(i)]{One has $ M \in \mathcal{M}^C_{\emph{\stf}} $ if and only if $ \emph{Hom}_R(M,E) \in \mathcal{M}^C_{\emph{\scot}} $.}
		\item[(ii)]{One has $ M \in \mathcal{M}^C_{\emph{\cf}} $ $ ( $resp. $ M \in \mathcal{M}^C_{\emph{\scf}} $$ ) $ if and only if $ \emph{Hom}_R(M,E) \in \mathcal{M}^C_{\emph{\ci}} $ $ ( $resp. $ \emph{Hom}_R(M,E) \in \mathcal{M}^C_{\emph{\sci}} $$ ) $.}
	\end{itemize}
\end{lem}
\begin{prop}
There is an inclusion $ \mathcal{M}^C_{\emph{\scp}} \subseteq \mathcal{M}^C_{\emph{\scf}} $. Also if $ M \in \mathcal{M}^C_{\emph{\scf}} $ is finitely generated, then $ M \in \mathcal{M}^C_{\emph{\scp}} $.
\end{prop}
\begin{proof}
Let $E$ be an injective cogenerator and let $ M \in \mathcal{M}^C_{\scp} $. By \cite[Theorem 3.2.1]{EJ1}, there is an isomorphism \\
\centerline{$ \Hom_R(\Tor^R_i(\mathcal{I}_C , M) , E) \cong \Ext^i_R(M , \Hom_R(\mathcal{I}_C , E)) $,}
for all $ i \geq 1 $. Now use the fact that $ \Hom_R(\mathcal{I}_C , E) \subseteq \mathcal{F}_C $ to obtain $ M \in \mathcal{M}^C_{\scf} $. Next suppose that $ M \in \mathcal{M}^C_{\scf} $ is finitely generated. The isomorphism \cite[Theorem 3.2.13]{EJ1}, \\
\centerline{$ \Hom_R(\Ext^i_R(M , \mathcal{F}_C) , E) \cong \Tor^R_i( \Hom_R(\mathcal{F}_C , E) , M) $,}
 for all $ i \geq 1 $, together with the fact that $ \Hom_R(\mathcal{F}_C , E) \subseteq \mathcal{I}_C $ yield $ M \in \mathcal{M}^C_{\scp} $.	
\end{proof}

\begin{prop}
Let $ M \in \mathcal{A}_C(R) $. The following are equivalent:
	\begin{itemize}
		\item[(i)]{ $ M $ is $C$-copure injective.}
		\item[(ii)] {If $0 \rightarrow M \rightarrow X \overset{\alpha}\rightarrow L \rightarrow 0$ is an exact sequence, with $ X \in \mathcal{I}_C $, then $ \alpha $ is $ \mathcal{I}_C $-precover.}
		\item[(i)]{ $ M $ is a kernel of a surjective $ \mathcal{I}_C $-precover $ \alpha : N \rightarrow L $.}
	\end{itemize}
\end{prop}
\begin{proof}
(i) $ \Longrightarrow $ (ii). Application of the functor $ \Hom_R(\mathcal{I}_C , -) $ on the exact sequence $0 \rightarrow M \rightarrow X \overset{\alpha}\rightarrow L \rightarrow 0$ yields an exact sequence \\
\centerline{$ \Hom_R(\mathcal{I}_C , X) \overset{\Hom_R(\mathcal{I}_C,\alpha)}\longrightarrow \Hom_R(\mathcal{I}_C, L) \longrightarrow \Ext^1_R(\mathcal{I}_C , M) = 0$.}
Hence $ \alpha $ is $ \mathcal{I}_C $-precover.

(ii) $ \Longrightarrow $ (iii). As $ M \in \mathcal{A}_C(R) $, by Lemma 2.10(ii), the $ \mathcal{I}_C $-precover of $M$ is injective. Thus there exists an exact sequence $0 \rightarrow M \rightarrow X \rightarrow L \rightarrow 0$, with $ X \in \mathcal{I}_C $. Now by (ii), $ X \rightarrow L $ is $ \mathcal{I}_C $-precover whose kernel is $M$.

(iii) $ \Longrightarrow $ (i). By assumption, there is an exact
 sequence $ 0 \rightarrow M \rightarrow X \overset{\beta}\rightarrow L \rightarrow 0 $ is which $ \beta $ is $ \mathcal{I}_C $-precover.
Application of the functor $ \Hom_R(\mathcal{I}_C , -) $ on this sequence yields an exact sequence \\
\centerline{$ \Hom_R(\mathcal{I}_C , X) \overset{\Hom_R(\mathcal{I}_C,\beta)}\longrightarrow \Hom_R(\mathcal{I}_C, L) \longrightarrow \Ext^1_R(\mathcal{I}_C , M) \longrightarrow \Ext^1_R(\mathcal{I}_C , X) = 0$.}
 But $ \Hom_R(\mathcal{I}_C,\beta) $ is surjective since
	$ \beta $ is $ \mathcal{I}_C $-precover. So that $ \Ext^1_R(\mathcal{I}_C , M) = 0$, whence $ M $ is $C$-copure injective.
\end{proof}

\begin{prop}
	Let $ M \in \mathcal{B}_C(R) $. The following are equivalent:
	\begin{itemize}
		\item[(i)]{ $ M $ is $C$-copure projective.}
		\item[(ii)] {If $0 \rightarrow L \overset{\alpha}\rightarrow Y \rightarrow M \rightarrow 0$ is an exact sequence, with $ Y \in \mathcal{P}_C $, then $ \alpha $ is $ \mathcal{P}_C $-preenvelope.}
		\item[(i)]{ $ M $ is a cokernel of an injective $ \mathcal{P}_C $-preenvelope $ \alpha : N \rightarrow L $.}
	\end{itemize}
\end{prop}
\begin{proof}
	Is dual of the proof of Proposition 4.7.
\end{proof}
By now, we know that $ \mathcal{GI}_C \subseteq  \mathcal{M}^C_{\sci} $. In the following, we show that the equality needs an additional assumption.
\begin{prop}
	Assume that each element of $ \mathcal{M}^C_{\emph{\sci}} $ has an epic $ \mathcal{I}_C $-cover. The following statements hold true:
	\begin{itemize}
		\item[(i)]{ One has $ \mathcal{M}^C_{\emph{\sci}} = \mathcal{GI}_C $.}
		\item[(ii)] {Assume further that $R$ has a dualizing module $D$. Then $ \mathcal{M}^C_{\emph{\sci}} \subseteq \mathcal{B}_{C^{\dagger}}(R) $, where $ C^{\dagger} = \emph{\Hom}_R(C,D) $. }
	\end{itemize}	
\end{prop}
\begin{proof}
	(i). The inclusion $ \mathcal{GI}_C \subseteq \mathcal{M}^C_{\sci} $ is from the definition. Assume that $ M \in \mathcal{M}^C_{\sci} $. Then $ \Ext^i_R(\mathcal{I}_C,M) = 0 $ for all $ i \geq 1 $. Hence we need only to show that
	$ M $ has a left $ \mathcal{I}_C $-resolution that is exact and $ \Hom_R(\mathcal{I}_C , -) $-exact. By the assumption,
	there exists an exact sequence \\
	\centerline{$ 0 \rightarrow K \rightarrow E \rightarrow M \rightarrow 0$,}	
	in which $ E \in \mathcal{I}_C $ and $ K = \ker (E \rightarrow M) $. Application of the functor $ \Hom_R(\mathcal{I}_C , -) $ on this sequence yields $ \Ext^i_R(\mathcal{I}_C,K) = 0 $ for all $ i \geq 2 $. On the other hand, the Wakamatsu's lemma \cite[Corollary 7.2.3]{EJ1} implies that
	$ \Ext^1_R(\mathcal{I}_C,K) = 0 $. Hence $ K \in \mathcal{M}^C_{\sci} $. Therefore we can repeat this argument to give an exact left
	 $ \mathcal{I}_C $-resolution for $ M $ which is $ \Hom_R(\mathcal{I}_C , -) $-exact, whence $ M \in \mathcal{GI}_C $.
	
	(ii). Follows from (i) and \cite[Theorem 4.6]{HJ}.
\end{proof}

Recall from \cite{EJ5} that an $R$-module $M$ is said to be $h$-\textit{divisible}, precisely when $M$ is homomorphic image of an injective $R$-module.
\begin{cor}
	Assume that each element of $ \mathcal{M}_{\emph{\sci}} $ is $h$-divisible. Then $ \mathcal{M}_{\emph{\sci}} = \mathcal{GI} $. Moreover if $C$ is dualizing, then $ \mathcal{M}_{\emph{\sci}} \subseteq \mathcal{B}_C (R) $.
\end{cor}

\begin{prop}
The following statements hold true:
	\begin{itemize}
		\item[(i)]{ Assume that $ M \in \mathcal{B}_C (R) $. One has $ M \in \mathcal{M}^C_{\emph{\scp}} $ if and only if $ \emph{\Hom}_R(C,M) \in \mathcal{M}_{\emph{\scp}} $.}
		\item[(ii)]{ Assume that $ M \in \mathcal{A}_C (R) $. One has $ M \in \mathcal{M}_{\emph{\scp}} $ if and only if $ C \otimes_R M \in \mathcal{M}^C_{\emph{\scp}} $.}
		\item[(iii)]{ Assume that $ M \in \mathcal{B}_C (R) $. One has $ M \in \mathcal{M}^C_{\emph{\scf}} $ if and only if $ \emph{\Hom}_R(C,M) \in \mathcal{M}_{\emph{\scf}} $.}
		\item[(iv)] {Assume that $ M \in \mathcal{A}_C (R) $. One has $ M \in \mathcal{M}_{\emph{\scf}} $ if and only if $ C \otimes_R M \in \mathcal{M}^C_{\emph{\scf}} $.}
		\item[(v)] {Assume that $ M \in \mathcal{B}_C (R) $. One has $ M \in \mathcal{M}_{\emph{\sci}} $ if and only if $ \emph{\Hom}_R(C,M) \in \mathcal{M}^C_{\emph{\sci}} $.}	
		\item[(vi)] {Assume that $ M \in \mathcal{A}_C (R) $. One has $ M \in \mathcal{M}^C_{\emph{\sci}} $ if and only if $ C \otimes_R M \in \mathcal{M}_{\emph{\sci}} $.}
		\end{itemize}	
\end{prop}	
\begin{proof} We only prove (i). The other statements are similar. Let $ F $ be a flat $R$-module. One has the isomorphisms
	 \[\begin{array}{rl}
	 \Ext^i_R(M, C \otimes_R F) &\cong \Ext^i_R(C \otimes_R \Hom_R(C,M), C \otimes_R F)\\
	 &\cong \Ext^i_{\mathcal{I}_C}(\Hom_R(C,M), F)\\
	 &\cong \Ext^i_R(\Hom_R(C,M), F)\\
	 \end{array}\]	
	 in which the first isomorphism holds because $ M \in \mathcal{B}_C(R) $, the second isomorphism is from \cite[Theorem 4.1]{TW}, and the last isomorphism is from \cite[Corollary 4.2(b)]{TW} since $ \Hom_R(C,M) \in \mathcal{A}_C(R) $ by \cite[Theorem 2.8(a)]{TW}. It follows that $ M \in \mathcal{M}^C_{\scp} $ if and only if $ \Hom_R(C,M) \in \mathcal{M}_{\scp} $, as wanted.
\end{proof}
\begin{cor}
The following statements hold true:
\begin{itemize}
	\item[(i)]{Let $ \mathcal{M}_{\emph{\scp}} \subseteq \mathcal{A}_C (R) $. There is an equivalence of categories
				 $ \mathcal{M}_{\emph{\scp}} \rightleftarrows \mathcal{M}^C_{\emph{\scp}} $ by the functors $ C \otimes_R - : \mathcal{M}_{\emph{\scp}}\rightarrow \mathcal{M}^C_{\emph{\scp}} $ and $ \emph{\Hom}_R(C,-) : \mathcal{M}^C_{\emph{\scp}} \rightarrow \mathcal{M}_{\emph{\scp}} $.}
	\item[(ii)]{Let $ \mathcal{M}_{\emph{\scf}} \subseteq \mathcal{A}_C (R) $. There is an equivalence of categories $ \mathcal{M}_{\emph{\scf}} \rightleftarrows \mathcal{M}^C_{\emph{\scf}} $ by the functors $ C \otimes_R - : \mathcal{M}_{\emph{\scf}}\rightarrow \mathcal{M}^C_{\emph{\scf}} $ and $ \emph{\Hom}_R(C,-) : \mathcal{M}^C_{\emph{\scf}} \rightarrow \mathcal{M}_{\emph{\scf}} $.}
	\item[(iii)]{Let $ \mathcal{M}_{\emph{\sci}} \subseteq \mathcal{B}_C (R) $. There is an equivalence of categories $ \mathcal{M}_{\emph{\sci}} \rightleftarrows \mathcal{M}^C_{\emph{\sci}} $ by the functors $ \emph{\Hom}_R(C,-) : \mathcal{M}_{\emph{\sci}}\rightarrow \mathcal{M}^C_{\emph{\sci}} $ and $ C \otimes_R - : \mathcal{M}^C_{\emph{\sci}} \rightarrow \mathcal{M}_{\emph{\sci}} $.}
\end{itemize}
\end{cor}  
\begin{proof}
Just use Proposition 4.11 and \cite[Theorem 2.8]{TW}.
\end{proof}
\begin{thm}
Let $M$ be an $R$-module. The following statements hold true:
\begin{itemize}
	\item[(i)]{One has $ M \in \mathcal{I}_C $ if and only if $ M \in \mathcal{M}^C_{\emph{\sci}}$ and $ C \emph{-\id}_R(M) < \infty$.}
	\item[(ii)]{One has $ M \in \mathcal{P}_C $ if and only if $ M \in \mathcal{M}^C_{\emph{\scp}}$ and $ C \emph{-\pd}_R(M) < \infty$.}
\end{itemize}	
\end{thm}
\begin{proof}
	(i). The necessity is trivial. We only need to prove the sufficiency. Since $C$-$\id_R(M) < \infty$, we have $ M \in \mathcal{A}_C(R) $. Now by Proposition 4.11(v), we have $ C \otimes_R M \in \mathcal{M}_{\sci} $ and hence $ \Ext^i_R(\mathcal{I} , C \otimes_R M ) = 0 $ for all $ i \geq 1 $. Moreover, by Lemma 2.10(ii), $ M $ has an exact $ \mathcal{I}_C $-coresolution which is $ (C \otimes_R -) $-exact. Hence $M$ has a bounded injective resolution \\
\centerline{$ 0 \rightarrow C \otimes_R M \rightarrow E^0 \rightarrow E^1 \rightarrow \cdots \rightarrow E^n \rightarrow 0 $.}
Assume that $ L^i $ is the $ i $-th kernel in the above coresolution. Then $ L^i  \in \mathcal{M}_{\sci} $ since the class $ \mathcal{M}_{\sci} $ is coresolving. In particular $ \Ext^1_R(E^n , L^{n-1}) = 0 $, which shows that the exact sequence
$ 0 \rightarrow L^{n-1} \rightarrow E^{n-1} \rightarrow E^n \rightarrow 0  $ is split. Consequently, $ L^{n-1} $ is injective. Repeating this argument, shows that $ L^1 $  is injective and then we conclude that exact sequence
$ 0 \rightarrow C \otimes_R M \rightarrow E^0 \rightarrow L^1 \rightarrow 0  $ is split. Therefore $ C \otimes_R M $ is injective and so
$ M \in \mathcal{I}_C $ by Lemma 2.5(i).

(ii). Is dual to the proof of (i).
\end{proof}

\begin{cor}
The following statements hold true:
	\begin{itemize}
		\item[(i)]{ An $R$-module $M$ is injective if and only if $ M \in \mathcal{M}_{\emph{\sci}}$ and $\emph{\id}_R(M) < \infty$.}
		\item[(ii)]{An $R$-module $N$ is projective if and only if $ M \in \mathcal{M}_{\emph{\scp}}$ and $\emph{\pd}_R(M) < \infty$.}
	\end{itemize}	
\end{cor}

\begin{prop}
	Let $M$ be an $R$-module. The following statements hold true.
	\begin{itemize}
		\item[(i)]{One has $ M \in \mathcal{M}^C_{\emph{\ci}} $ if and only if $ \emph{\Hom}_R(\mathcal{F} , M)
			\subseteq \mathcal{M}^C_{\emph{\ci}}$.}
		\item[(ii)]{One has $ M \in \mathcal{M}^C_{\emph{\cp}} $ if and only if $ \underline{\mathcal{P}} \otimes_R M
			\subseteq \mathcal{M}^C_{\emph{\cp}}$, where $ \underline{\mathcal{P}} $ is the subcategory of finitely generated projective modules.}
		\item[(iii)]{One has $ M \in \mathcal{M}^C_{\emph{\scf}} $ if and only if $ \mathcal{P} \otimes_R M
			\subseteq \mathcal{M}^C_{\emph{\scf}}$.}
		\item[(iv)] { If $ M \in \mathcal{M}^C_{\emph{\scot}} \cap \mathcal{M}_{\emph{\sci}}$, then $ \emph{\Hom}_R(\mathcal{F}_C , M)
			\subseteq \mathcal{M}^C_{\emph{\sci}}$.}
	\end{itemize}
\end{prop}	
\begin{proof}
	(i). Let $E$ be an injective $R$-module and let $F$ be a flat $R$-module. There exists an exact sequence \\
	\centerline{$ 0 \rightarrow K \rightarrow P \rightarrow \Hom_R(C,E) \rightarrow 0 $, $ (*) $}
	in which $P$ is projective and $ K = \ker(P \rightarrow \Hom_R(C,E) ) $. Applying the exact functor $ - \otimes_R F $ on $ (*) $, we get an exact sequence \\
	 	\centerline{$ 0 \rightarrow K \otimes_R F \rightarrow P \otimes_R F \rightarrow \Hom_R(C,E) \otimes_R F \rightarrow 0 $. $ (**) $}
	Note that $ \Hom_R(C,E) \otimes_R F \cong \Hom_R(C,E \otimes_R F) $ by \cite[Theorem 3.2.14]{EJ1}. Therefore $ \Hom_R(C,E) \otimes_R F \in \mathcal{I}_C $. Now assume that $ M \in \mathcal{M}^C_{\ci} $. Applying the functor $ \Hom_R(-,M) $ on $ (**) $, we get the exact sequence \\
		 	\centerline{$ 0 \rightarrow \Hom_R(P \otimes_R F , M) \rightarrow \Hom_R(K \otimes_R F , M) \rightarrow \Ext_R^1(\Hom_R(C,E \otimes_R F) , M) =0$.}
		 	Finally, the Hom-tensor adjointness isomorphism yields the exact sequence \\
		 	\centerline{$ \Hom_R(P ,\Hom_R(F, M)) \rightarrow \Hom_R(K ,\Hom_R(F, M)) \rightarrow \Ext_R^1(\Hom_R(C,E) , \Hom_R(F, M))$,}
		 	from which we conclude that $ \Ext_R^1(\Hom_R(C,E) , \Hom_R(F, M)) = 0 $. Hence $ \Hom_R(F, M)
		 	\in \mathcal{M}^C_{\ci}$. The converse is evident.	
	
		(ii). Let $F$ be an flat $R$-module and let $P$ be a finitely generated projective $R$-module. There exists an exact sequence \\
		\centerline{$ 0 \rightarrow C \otimes_R F \rightarrow E \rightarrow L \rightarrow 0 $ , $ (\dagger) $}
		in which $E$ is injective and $ L = \coker(C \otimes_R F \rightarrow E) $. Applying the exact functor $ \Hom_R(P,-) $ on $ (\dagger) $, we get an exact sequence \\
		\centerline{$ 0 \rightarrow \Hom_R(P, C \otimes_R F) \rightarrow \Hom_R(P,E) \rightarrow \Hom_R(P,L) \rightarrow 0 $. $ (\dagger \dagger) $}
		Note that $ \Hom_R(P , C \otimes_R F) \cong \Hom_R(P,F) \otimes_R C $ and therefore $ \Hom_R(P , C \otimes_R F) \in \mathcal{F}_C $. Now assume that $ M \in \mathcal{M}^C_{\cp} $. Applying the functor $ \Hom_R(M,-) $ on $ (\dagger \dagger) $, we get the exact sequence \\
		\centerline{$ \Hom_R(M , \Hom_R(P,E)) \rightarrow \Hom_R(M , \Hom_R(P,L)) \rightarrow \Ext_R^1(M , \Hom_R(P , C \otimes_R F)) =0$.}
		Finally, the Hom-tensor adjointness isomorphism yields the exact sequence \\
		\centerline{$ \Hom_R(P \otimes_R M , E) \rightarrow \Hom_R(P \otimes_R M ,L) \rightarrow \Ext_R^1(P \otimes_R M , C \otimes_R F)$,}
		from which we conclude that $ \Ext_R^1(P \otimes_R M , C \otimes_R F) = 0 $. Hence $ P \otimes_R M \in \mathcal{M}^C_{\cp}$. The converse is evident.
		
		(iii). Just observe that $ \Tor_i^R(\mathcal{I}_C , \mathcal{F} \otimes_R M) \cong \Tor_i^R(\mathcal{I}_C  \otimes_R \mathcal{F} , M)  $ by \cite[Corollary 10.61]{R}, and that $ \mathcal{I}_C  \otimes_R \mathcal{F} \subseteq  \mathcal{I}_C$ by \cite[Theorem 3.2.14]{EJ1}.

(iv). First note that $ \Tor^R_i(\mathcal{I}_C , \mathcal{F}_C) = 0 $ for all $ i > 0 $. Hence by \cite[Theorem 10.64]{R} there is a third quadrant spectral sequence  \\
\centerline{$ E_2^{p,q} = \Ext^p_R(\mathcal{I}_C , \Ext^q_R(\mathcal{F}_C , M)) \underset{p}\Rightarrow \Ext^{p+q}_R (\mathcal{I}_C \otimes_R \mathcal{F}_C , M)$.}
Next observe that  $\Ext^{q>0}_R(\mathcal{F}_C , M) = 0$ by assumption, and there is an inclusion $ \mathcal{I}_C \otimes_R \mathcal{F}_C \subseteq \mathcal{I} $. Hence $ \Ext^i_R(\mathcal{I}_C , \Hom_R(\mathcal{F}_C , M)) \cong \Ext^i_R (\mathcal{I}_C \otimes_R \mathcal{F}_C , M) = 0 $ for all $ i \geq 1 $, as wanted.	
\end{proof}
\begin{thm}
The following statements hold true:
\begin{itemize}
	\item[(i)]{ Assume that $ M \in \mathcal{M}^C_{\emph{\scf}} $. Then $ \mathcal{I}_C \otimes_R M \subseteq \mathcal{M}_{\emph{\scot}}$.}
	\item[(ii)] {Assume that $ M \in \mathcal{M}^C_{\emph{\sci}} $. Then $ \emph{\Hom}_R(\mathcal{I}_C , M) \subseteq \mathcal{M}_{\emph{\stf}}$.}
\end{itemize}	
\end{thm}

\begin{proof}
	(i). By assumption $ \Tor^R_i(\mathcal{I}_C, M) = 0  $ for all $ i \geq 1 $. Hence if $ \textbf{P} \rightarrow M $ is a projective resolution of $M$ then $ \mathcal{I}_C \otimes_R \textbf{P} \rightarrow \mathcal{I}_C \otimes_R M \rightarrow 0$ is an exact complex which is an
	$ \mathcal{I}_C $-resolution for $ M $. Note that $ \mathcal{I}_C \subseteq \mathcal{M}_{\scot} $ by Proposition 3.1(i). Now break this
	$ \mathcal{I}_C $-resolution to short exact sequences and use \cite[Corollary 3.2.7]{GR}, to see that $ \Ext^i_R(F , \mathcal{I}_C \otimes_R M) = 0 $ for all $ i \geq 1 $ and all $R$-modules $F$ with $ \fd_R(F) < \infty $.
	
	(ii). By assumption $ \Ext^i_R(\mathcal{I}_C, M) = 0  $ for all $ i \geq 1 $. Hence if $ M \rightarrow \textbf{E} $ is an injective coresolution of $M$, then $0 \rightarrow \Hom_R(\mathcal{I}_C , M) \rightarrow \Hom_R(\mathcal{I}_C , \textbf{E})$ is an exact complex which is an
	$  \mathcal{F}_C  $-coresolution for $ \Hom_R(\mathcal{I}_C , M) $. Note that $ \mathcal{F}_C \subseteq \mathcal{M}_{\stf} $ by Proposition 3.1(i).
	Now break this $ \mathcal{F}_C $-coresolution to short exact sequences to see that $ \Tor^R_i(F , \Hom_R(\mathcal{I}_C , M)) = 0 $ for all $ i \geq 1 $ and all $R$-modules $F$ with $ \fd_R(F) < \infty $.
\end{proof}
\begin{thm}	
Let $\emph{\dim}(R) < \infty$ and let $C$ be dualizing. The following hold true:
	\begin{itemize}
		\item[(i)] {$ \mathcal{M}^C_{\emph{\scot}} = \mathcal{M}_{\emph{\sci}} $.}
		\item[(ii)] {$ \mathcal{M}_{\emph{\scot}} = \mathcal{M}^C_{\emph{\sci}} $.}	
		\item[(iii)] {$ \mathcal{M}^C_{\emph{\stf}} = \mathcal{M}_{\emph{\scf}} $.}
		\item[(iv)] {$ \mathcal{M}_{\emph{\stf}} = \mathcal{M}^C_{\emph{\scf}} $.}
		\item[(v)] {If, in addition, $(R , \fm)$  is local, then $ \underline{\mathcal{M}}^C_{\emph{\scp}} = \emph{\mcm}(R) $ where $ \underline{\mathcal{M}}^C_{\emph{\scp}} $ is the full subcategory of $ \mathcal{M}^C_{\emph{\scp}} $ consisting of finitely generated $R$-modules and $ \emph{\mcm}(R) $ is the full subcategory of maximal Cohen-Macaulay $R$-modules.}	
	\end{itemize}	
\end{thm}
\begin{proof}
(i) First assume that $ M \in \mathcal{M}^C_{\scot} $. Assume that $E$ is an $R$-module with $\id_R(M) < \infty$.
Then $ E \in \mathcal{B}_C (R) $ and hence by Lemma 2.10(ii), the $ \mathcal{F}_C $-precover of $M$ is surjective. So that we have an exact sequence of the form \\
\centerline{$ 0 \rightarrow L \rightarrow X \rightarrow M \rightarrow 0$ ,}
in which $X \in \mathcal{F}_C $ and $ L = \ker(X \rightarrow M) $. Note that $ \id_R(X) < \infty $ since $C$ is dualizing. Hence $\id_R(L) < \infty$. Therefore by this argument for any $ i \geq 0 $ we can construct exact sequences \\
 \centerline{$ 0 \rightarrow E_{i + 1} \rightarrow X_i \rightarrow E_i \rightarrow 0$ , $ (*) $}
 in which $ E_0 = E $, $X_i \in \mathcal{F}_C $ and $ \id_R(E_i) < \infty $. Since $ E_{i+1} \in \mathcal{B}_C (R) $ we have
 $  \Ext^1_R(C , E_{i+1} ) = 0 $. Hence the exact sequence \\
  \centerline{$ 0 \rightarrow E_{i + 1} \rightarrow X_i \rightarrow E_i \rightarrow 0$ ,}
  induces an exact sequence \\
   \centerline{$ 0 \rightarrow \Hom_R(C , E_{i + 1}) \rightarrow \Hom_R(C , X_i) \rightarrow \Hom_R(C , E_i) \rightarrow 0$ .}
   Now by Lemma 2.5(iii), $ \Hom_R(C , X_i) $ is flat and
    also \\
\centerline{ $ C $-$ \fd_R(E_{i + 1}) = \fd_R(\Hom_R(C , E_{i + 1})) = \fd_R(\Hom_R(C , E_i)) - 1 = C$-$ \fd_R(E_i) - 1 $,}
for all $ i \geq 0 $. Consequently, $ E_j \in \mathcal{F}_C $ for all $ j \geq d $. Application of the functor $ \Hom_R(- , M) $ on the exact sequences $ (*) $ yields the isomorphisms \\
 \centerline{$ \Ext^i_R(E_0 , M) \cong \Ext^{i+1}_R(E_1 , M) \cong \cdots \cong \Ext^{i+d}_R(E_d , M) = 0 $,}
for all $ i \geq 0 $, where the last equality holds because $ M \in \mathcal{M}^C_{\scot} $. Hence $ M \in \mathcal{M}_{\sci} $. The inclusion
$ \mathcal{M}^C_{\scot} \supseteq \mathcal{M}_{\sci} $ is trivial since all modules in $ \mathcal{F}_C  $ have finite injective dimensions.

(ii). First assume that $ M \in \mathcal{M}^C_{\sci} $. Assume that $F$ is an $R$-module with $\fd_R(M) < \infty$.
Then $ F \in \mathcal{A}_C (R) $ and hence by Lemma 2.10(ii), the $ \mathcal{I}_C $-preenvelope of $M$ is injective. Note that all modules in
 $ \mathcal{I}_C $ have finite flat dimensions since $ C $ is dualizing. Therefore we can construct exact sequences \\
 \centerline{$ 0 \rightarrow F_i \rightarrow Y_i \rightarrow F_{i + 1} \rightarrow 0$ , $ (**) $}
 in which $ F_0 = F $, $Y_i \in \mathcal{I}_C $ and $ \fd_R(F_i) < \infty $. As $ \fd_R(F_d) < \infty $, by \cite[Corollary 3.4]{F1}, we have $\pd_R (F_d) \leq d$. Application of the functor $ \Hom_R(- , M) $ on the exact sequences $ (**) $ yields the isomorphisms \\
 \centerline{$ \Ext^i_R(F_0 , M) \cong \Ext^{i+1}_R(F_1 , M) \cong \cdots \cong \Ext^{i+d}_R(F_d , M) = 0 $.}
 for all $ i \geq 1 $. Thus $ M \in \mathcal{M}_{\scot} $. The inclusion
 $ \mathcal{M}_{\scot} \subseteq \mathcal{M}^C_{\sci} $ is trivial since all modules in $ \mathcal{I}_C $ have finite flat dimensions.

 (iii). Assume that $E$ is an injective cogenerator and that $ M \in \mathcal{M}^C_{\stf} $.
  Then $ \Hom_R(M,E) \in \mathcal{M}^C_{\scot} = \mathcal{M}_{\sci} $ by Lemma 4.5(i). Hence $ M \in \mathcal{M}_{\scf} $ by Lemma 2.15(ii). Next
  if $ M \in \mathcal{M}_{\scf} $, then $ \Hom_R(M,E) \in \mathcal{M}_{\sci} = \mathcal{M}^C_{\scot} $ by by Lemma 2.15(ii). Hence $ M \in  M \in \mathcal{M}^C_{\stf} $ by Lemma 4.5(i).

  (iv). Assume that $E$ is an injective cogenerator and that $ M \in \mathcal{M}^C_{\scf} $.
  Then $ \Hom_R(M,E) \in \mathcal{M}^C_{\sci} = \mathcal{M}_{\scot} $ by Lemma 4.5(ii). Hence $ M \in \mathcal{M}_{\stf} $ by Lemma 2.15(i). Next
  if $ M \in \mathcal{M}_{\stf} $, then $ \Hom_R(M,E) \in \mathcal{M}_{\scot} = \mathcal{M}^C_{\sci} $ by Lemma 2.15(i). Hence $ M \in \mathcal{M}^C_{\scf} $ by Lemma 4.5(ii).

  (v). Suppose that $ M \in \mcm(R) $. Then by \cite[Theorem 9.2.16]{EJ1}, we have $ \Ext^i_R(M,C) = 0 $ for all $i \geq 0$. Hence if $F$ is any flat $R$-module, then by \cite[Theorem 3.2.15]{EJ1} we have the isomorphism $ \Ext^i_R(M,C \otimes_R F) \cong \Ext^i_R(M,C) \otimes_R F = 0  $ for all $i \geq 0$, whence $ M \in \underline{\mathcal{M}}^C_{\scp} $. The reverse inclusion is easy by another use of \cite[Theorem 9.2.16]{EJ1}.
\end{proof}
\begin{cor}
Let $ \emph{\dim(R)} < \infty $ and let $C$ be dualizing. For an $R$-module $ M $, the following statements hold true:
	\begin{itemize}
		\item[(i)]{One has $ M \in \mathcal{I}_C $ if and only if $ M \in \mathcal{M}_{\emph{\scot}}$ and $ C \emph{-\id}_R(M) < \infty$.}
		\item[(ii)]{One has $ M \in \mathcal{F}_C $ if and only if $ M \in \mathcal{M}_{\emph{\stf}}$ and $ C \emph{-\fd}_R(M) < \infty$.}
		\item[(i)]{One has $ M \in \mathcal{I}_C $ if and only if $ M \in \mathcal{M}_{\emph{\scot}}$ and $ \emph{\fd}_R(M) < \infty$.}
		\item[(ii)]{One has $ M \in \mathcal{F}_C $ if and only if $ M \in \mathcal{M}_{\emph{\stf}}$ and $ \emph{\id}_R(M) < \infty$.}
	\end{itemize}	
\end{cor}
\begin{proof}
	(i). The necessity is trivial. We only need to prove the sufficiency. Assume that $ M \in \mathcal{M}_{\scot}$. Then by Theorem 4.17(ii),
	we have  $ M \in \mathcal{M}^C_{\sci}$, and hence $ M \in \mathcal{I}_C $ by Theorem 4.13(i).
	
	(ii). The necessity is trivial. Assume that $E$ is an injective cogenerator and that $ M \in \mathcal{M}_{\stf}$. Then
	$ \Hom_R(M,E) \in \mathcal{M}_{\scot} $ by Lemma 2.15(i), and that we have $ C $-$ \id_R(\Hom_R(M,E)) < \infty $ since $ \Hom_R(\mathcal{F}_C,E) \in \mathcal{I}_C $. Hence $ \Hom_R(M,E) \in \mathcal{I}_C $ by (i). Set $ \Hom_R(M,E) = \Hom_R(C,I) $ where $I$ is injective. We have the isomorphisms
	 \[\begin{array}{rl}
	 \Hom_R(\Hom_R(C,M) , E) &\cong C \otimes_R \Hom_R(M,E),\\
	 &\cong C \otimes_R \Hom_R(C, I)\\
	  &\cong I\\
	 \end{array}\]
where the first isomorphism is from \cite[theorem 3.2.11]{EJ1}, and the last one holds because $ I \in \mathcal{B}_C(R) $. It follows that $ \Hom_R(C,M) $ is flat, whence $ M \in \mathcal{F}_C $ by Lemma 2.5(iii).

(iii).  The necessity is trivial. If $ \fd_R(M) < \infty$, then $ M \in \mathcal{A}_C(R) $ and hence, by Lemma 2.10(ii), there is an exact sequence 
$ 0 \rightarrow M \rightarrow \Hom_R(C,I) \rightarrow L \rightarrow 0 $, in which $I$ is injective and $ L = \coker(M \rightarrow \Hom_R(C,I)) $. Now $ \fd_R(L) < \infty $ since both $M$ and $ \Hom_R(C,I) $ have finite flat dimensions. Hence $ \Ext_R^1(L,M) = 0 $, and the sequence splits, whence $ M \in \mathcal{I}_C $.

(iv). Again, the necessity is trivial. Now if $E$ is an injective cogenerator and $ \id_R(M) < \infty$, then $ \fd_R(\Hom_R(M,E)) < \infty $, and $ \Hom_R(M,E) \in \mathcal{M}_{\scot} $ by Lemma 2.15(i). Hence $ \Hom_R(M,E) \in \mathcal{I}_C $ by (iii), and then $ M \in \mathcal{F}_C $ by the argument used in part (ii).	
\end{proof}
\begin{cor}
	Let $R$ be a Gorenstein ring of finite Krull dimension. The following statements hold true:
	\begin{itemize}
		\item[(i)]{An $R$-module $M$ is injective if and only if $ M \in \mathcal{M}_{\emph{\scot}}$ and $ \emph{\id}_R(M) < \infty$.}
		\item[(ii)]{An $R$-module $M$ is flat if and only if $ M \in \mathcal{M}_{\emph{\stf}}$ and $ \emph{\fd}_R(M) < \infty$.}
	\end{itemize}	
\end{cor}
\begin{defn}
\emph{Let $ M $ be an $ R $-module. We say that an $R$-module $ M $ has $ C $-\textit{copure injective dimension} at most $ n $, denoted
	$ C $-$ \cid_R(M) \leq n $, if there is an exact sequence of $ R $-modules $ 0 \rightarrow M \rightarrow E^0 \rightarrow E^1 \rightarrow \cdots \rightarrow E^n \rightarrow 0 $ with each $ E^i $ strongly $ C $-copure
	injective. If there is no shorter such sequence, we set $ C $-$ \cid_R(M) = n$. Dually, we say that $ M $ has $ C $-\textit{copure flat dimension} (resp. $ C $-\textit{copure projective dimension}) at most $ m $, denoted
	$ C $-$ \cfd_R(M) \leq m $ (resp. $ C $-$ \cpd_R(M) \leq m $), if there is an exact sequence of $ R $-modules $ 0 \rightarrow F_m \rightarrow \cdots \rightarrow F_1 \rightarrow F_0  \rightarrow M \rightarrow 0 $ with each $ F_i $ strongly $ C $-copure
	flat (resp. strongly $ C $-copure
	projective). If there is no shorter such sequence, we set $ C $-$ \cfd_R(M) = m$ (resp. $ C $-$ \cpd_R(M) = m)$.}
\end{defn}

\begin{rem}
\emph{From the above definition, it easily follows that there are equalities $ C $-$ \cid_R(M) = \sup \{ n \geq 0 | \Ext^n_R(\Hom_R(C,E) , M) \neq 0 $, where $E$ is an injective $R$-module$\}$, $ C $-$ \cfd_R(M) = \sup \{ n \geq 0 | \Tor^R_n(\Hom_R(C,E) , M) \neq 0 $, where $E$ is an injective $R$-module$\}$ and $ C $-$ \cpd_R(M) = \sup \{ n \geq 0 | \Ext^n_R(M , C \otimes_R F) \neq 0 $, where $F$ is an flat $R$-module$\}$ for any $R$-module $M$. Also if $ 0 \rightarrow M \rightarrow N \rightarrow L \rightarrow 0 $ is an exact sequence of $ R $-modules and two of these three modules have finite $ C $-copure injective dimension (resp. $ C $-copure flat dimension and $ C $-copure projective dimension), then so has the third.}
\end{rem}
The following proposition compares the new homological dimension of Definition 4.20 with those which are appeared in \cite{TW} and \cite{HJ}.
\begin{prop}
	For any $R$-module $M$, the following statements hold true:
	\begin{itemize}
		\item[(i)]{ One has $ C\emph{-\cid}_R(M) \leq \emph{G}_C\emph{-\id}_R(M) \leq C\emph{-\id}_R(M) $. Moreover, the equality holds provided that $ C\emph{-\id}_R(M) < \infty $.}
		\item[(ii)] {One has $ C\emph{-\cpd}_R(M) \leq \emph{G}_C\emph{-\pd}_R(M) \leq C\emph{-\pd}_R(M) $. Moreover, the equality holds provided that $ C\emph{-\pd}_R(M) < \infty $.}
		\item[(iii)] {One has $ C\emph{-\cfd}_R(M) \leq \emph{G}_C\emph{-\fd}_R(M) \leq C\emph{-\fd}_R(M) $. Moreover, the equality holds provided that $ C\emph{-\fd}_R(M) < \infty $.}	
	\end{itemize}	
\end{prop}
\begin{proof}
	(i). The containments $ \mathcal{I}_C \subseteq \mathcal{GI}_C \subseteq \mathcal{M}^C_{\sci} $ show the inequalities. Now assume that
	$ C$-$\id_R(M) = n < \infty $. Hence by Lemma 2.5(i), we have $ \id_R(C \otimes_R M) = n $ and so there exists a non-zero $R$-module $K$ for which $ \Ext_R^n(K, C \otimes_R M) \neq 0 $. There exists an exact sequence \\
	\centerline{$ 0 \rightarrow K \rightarrow I \rightarrow L \rightarrow 0 $, $ (*) $}
	in which $I$ is injective and $ L = \coker(K \rightarrow I) $. We show that $  \Ext_R^n(\Hom_R(C,I), M) \neq 0  $. Note that both $ I $ and $ C \otimes_R M $ are in $ \mathcal{B}_C(R) $ and that $ M \cong \Hom_R(C , C \otimes_R M) $ since $ M \in \mathcal{A}_C(R) $. Hence we have the isomorphisms
	\[\begin{array}{rl}
	\Ext_R^n(\Hom_R(C,I), M) &\cong \Ext_R^n(\Hom_R(C,I), \Hom_R(C , C \otimes_R M))\\
	&\cong\Ext_{\mathcal{P}_C}^n(I , C \otimes_R M)\\
	&\cong\Ext_R^n(I , C \otimes_R M),\\
	\end{array}\]
	where the second isomorphism is from \cite[Theorem 4.1]{TW} and the last one is from \cite[Corollary 4.2]{TW}. Hence we need only to show that $ \Ext_R^n(I , C \otimes_R M) \neq 0 $. To do this, apply the functor $ \Hom_R(- , C \otimes_R M) $ to the exact sequence $ (*) $ to get an exact sequence \\
	\centerline{$ \Ext_R^n(I , C \otimes_R M) \rightarrow \Ext_R^n(K , C \otimes_R M) \rightarrow \Ext_R^{n+1}(L , C \otimes_R M) = 0 $.}
	This completes the proof.
	
	(ii). The containments $ \mathcal{P}_C \subseteq \mathcal{GP}_C \subseteq \mathcal{M}^C_{\scp} $ show the inequalities. Now assume that
	$ C$-$\pd_R(M) = n < \infty $. Hence by Lemma 2.5(ii), we have $ \pd_R(\Hom_R(C,M)) = n $ and so there exists a non-zero $R$-module $K$ for which $ \Ext_R^n(\Hom_R(C,M) , K) \neq 0 $. There exists an exact sequence \\
	\centerline{$ 0 \rightarrow L \rightarrow P \rightarrow K \rightarrow 0 $, $ (**) $}
	in which $P$ is projective and $ L = \ker(P \rightarrow K) $. We show that $  \Ext_R^n(M, C \otimes_R P) \neq 0  $. Note that both $ P $ and $ \Hom_R(C,M) $ are in $ \mathcal{A}_C(R) $ and that $ M \cong C \otimes_R \Hom_R(C, M) $ since $ M \in \mathcal{B}_C(R) $. Hence we have the isomorphisms
	\[\begin{array}{rl}
	\Ext_R^n(M, C \otimes_R P) &\cong \Ext_R^n(C \otimes_R \Hom_R(C, M) , C \otimes_R P)\\
	&\cong\Ext_{\mathcal{I}_C}^n(\Hom_R(C, M) , P)\\
	&\cong\Ext_R^n(\Hom_R(C, M) , P),\\
	\end{array}\]
	where the second isomorphism is from \cite[Theorem 4.1]{TW} and the last one is from \cite[Corollary 4.2]{TW}. Hence we need only to show that $ \Ext_R^n(\Hom_R(C, M) , P) \neq 0 $. To do this, apply the functor $ \Hom_R(\Hom_R(C, M) , -) $ to the exact sequence $ (**) $ to get an exact sequence \\
	\centerline{$ \Ext_R^n(\Hom_R(C, M) , P) \rightarrow \Ext_R^n(\Hom_R(C, M) , K) \rightarrow \Ext_R^{n+1}(\Hom_R(C, M) , L) = 0 $.}
	This completes the proof.
	
	(iii). The containments $ \mathcal{F}_C \subseteq \mathcal{GF}_C \subseteq \mathcal{M}^C_{\scf} $ show the inequalities. Now assume that
	$C$-$\fd_R(M) = n < \infty $. Hence by Lemma 2.5(iii), we have $ \fd_R(\Hom_R(C,M)) = n $ and so there exists a non-zero $R$-module $K$ for which $ \Tor^R_n(\Hom_R(C,M) , K) \neq 0 $. There exists an exact sequence \\
	\centerline{$ 0 \rightarrow K \rightarrow I \rightarrow L \rightarrow 0 $,$  (***) $}
	in which $I$ is injective and $ L = \coker(K \rightarrow I) $. We show that $  \Tor^R_n(\Hom_R(C,I), M) \neq 0  $. Note that $ I \in \mathcal{B}_C(R) $ and that $ \Hom_R(C,M) \in \mathcal{A}_C(R) $. Note also that $ M \cong C \otimes_R \Hom_R(C, M) $ since $ M \in \mathcal{B}_C(R) $. Hence we have the isomorphisms
	\[\begin{array}{rl}
	\Tor^R_n(\Hom_R(C,I), M) &\cong \Tor^R_n(\Hom_R(C,I) , C \otimes_R \Hom_R(C, M))\\
	&\cong\Tor^{\mathcal{F}_C}_n(I , \Hom_R(C, M))\\
	&\cong\Tor^R_n(I , \Hom_R(C, M)),\\
	\end{array}\]
	where the second isomorphism is from \cite[Theorem 3.6]{STWY} and the last one is from \cite[Proposition 4.3]{STWY}. Hence we need only to show that $ \Tor^R_n(I , \Hom_R(C, M)) \neq 0 $. To do this, apply the functor $ \Hom_R(C, M) \otimes_R - $ to the exact sequence $ (***) $ to get an exact sequence \\
	\centerline{$0 = \Tor^R_{n+1}(L , \Hom_R(C, M)) \rightarrow \Tor^R_n(K , \Hom_R(C, M)) \rightarrow \Tor^R_n(I , \Hom_R(C, M)) $.}
	This completes the proof.
\end{proof}

The following three propositions are standard.
\begin{prop}
	For an $ R $-module $ M $ and an integer $ n \geq 0 $, the
	following are equivalent:
	\begin{itemize}
		\item[(i)]{$ C\emph{-\cid}_R(M) \leq n $.}
		\item[(ii)]{Every n-th cosyzygy of $ M $ is strongly $ C $-copure injective}
		\item[(iii)]{There exists an exact sequence $ 0 \rightarrow M \rightarrow E^0 \rightarrow E^1 \rightarrow \cdots \rightarrow E^n \rightarrow 0 $ with $ E^i \in \mathcal{M}^C_{\emph{\sci}} $.}
	\end{itemize}
\end{prop}

\begin{prop}
	For an $ R $-module $ M $ and an integer $ n \geq 0 $, the
	following are equivalent:
	\begin{itemize}
		\item[(i)]{$ C\emph{-\cpd}_R(M) \leq n $.}
		\item[(ii)]{Every n-th syzygy of $ M $ is strongly $ C $-copure projective.}
		\item[(iii)]{There exists an exact sequence $ 0 \rightarrow P_n \rightarrow \cdots \rightarrow P_1 \rightarrow P_0 \rightarrow M \rightarrow 0 $ with $ P_i \in \mathcal{M}^C_{\emph{\scp}} $.}
	\end{itemize}
\end{prop}

\begin{prop}
	For an $ R $-module $ M $ and an integer $ n \geq 0 $, the
	following are equivalent:
	\begin{itemize}
		\item[(i)]{$ C\emph{-\cfd}_R(M) \leq n $.}
		\item[(ii)]{Every n-th syzygy of $ M $ is strongly $ C $-copure flat.}
		\item[(iii)]{There exists an exact sequence $ 0 \rightarrow F_n \rightarrow \cdots \rightarrow F_1 \rightarrow F_0 \rightarrow M \rightarrow 0 $ with $ F_i \in \mathcal{M}^C_{\emph{\scf}} $.}
	\end{itemize}
\end{prop}
In the following proposition, we show that for a finitely generated $R$-module, there is a special $R$-module for detecting its homological dimensions.
\begin{prop}
	Let $M$ be a finitely generated $R$-module and let $E$ be an injective cogenerator. The following statements hold true:
	\begin{itemize}
		\item[(i)]{If $ C\emph{-\cpd}_R(M) < \infty $, then $ C\emph{-\cpd}_R(M) = \sup \{ i \geq 0 \mid \emph{\Ext}^i_R(M,C) \neq 0 \} $.}
		\item[(ii)]{If $ C\emph{-\cfd}_R(M) < \infty $, then $  C\emph{-\cfd}_R(M) = \sup \{ i \geq 0 \mid \emph{\Tor}_i^R(\emph{\Hom}_R(C,E),M) \neq 0 \} $.}
		\item[(iii)]{If, in addition, $ (R, \fm) $ is complete local and $ C\emph{-\cid}_R(M) < \infty $, then \\
\centerline{ $  C\emph{-\cid}_R(M) = \sup \{ i \geq 0 \mid \emph{\Ext}^i_R(\emph{\Hom}_R(C,E(R/ \fm)),M) \neq 0 \}$.}}
	\end{itemize}	
\end{prop}

\begin{proof}
	(i). Suppose that $ C$-$\cpd_R(M) = n < \infty $. Then there exists a flat $R$-module $F$ for which $ \Ext^n_R(M, C \otimes_R F) \neq 0 $. But Since $M$ is finitely generated, by  \cite[Theorem 3.2.15]{EJ1}, we have the isomorphism \\
	\centerline{$ \Ext^n_R(M, C) \otimes_R F \cong\Ext^n_R(M, C \otimes_R F) \neq 0 $,}
	whence $ \Ext^n_R(M, C) \neq 0 $. This proves the inequality '$ \leq $'. The inequality '$ \geq $' is trivial.

	(ii). Suppose that $ C$-$\cfd_R(M) = n < \infty $. Then there exists an injective $R$-module $I$ for which $ \Tor_n^R(\Hom_R(C,I),M) \neq 0 $.
	But Since $M$ is finitely generated, we have the isomorphism \cite[Theorem 3.2.13]{EJ1} \\
	\centerline{$ \Hom(\Ext^n_R(M, C) , I) \cong \Tor_n^R(\Hom_R(C,I),M) \neq 0 $,}
	which implies that $ \Ext^n_R(M, C) \neq 0 $. Therefore using \cite[Theorem 3.2.13]{EJ1} once more, we get \\
	\centerline{$ \Tor_n^R(\Hom_R(C,E),M) \cong \Hom(\Ext^n_R(M, C) , E) \neq 0 $.}
	This proves the inequality '$ \leq $'. The inequality '$ \geq $' is trivial.
	
	(iii). Suppose that $ C$-$\cid_R(M) = n < \infty $. Then there exists an injective $R$-module $I$ for which $ \Ext^n_R(\Hom_R(C,I),M) \neq 0 $.
	We show that the maximal ideal $\fm$ occurs in the decomposition of $I$. There is an exact sequence (See the proof of \cite[Proposition 2.2]{FFGR}) \\
	\centerline{$  0 \rightarrow  \underset{ \underset {j \in J} \longleftarrow } \lim (M/ \fm^j M) \rightarrow \prod_{j \in J} (M/ \fm^j M) \rightarrow \prod_{j \in J} (M/ \fm^j M) \rightarrow 0 $. (*)}
	Since $M$ is finitely generated and hence is complete, we have $ M \cong \underset{ \underset {i \in J} \longleftarrow } \lim (M/ \fm^j M) $. Choose a prime ideal $ \fp \neq \fm $ and set $ X = \Hom_R(C,E(R/ \fp)) $. The sequence (*) yields an exact sequence \\
	\centerline{$ \cdots \rightarrow \Ext^i_R(X , \prod_{j \in J} (M/ \fm^j M)) \rightarrow \Ext^{i+1}_R(X , M) \rightarrow \Ext^{i+1}_R(X , \prod_{j \in J} (M/ \fm^j M)) \rightarrow \cdots $.}
	We know that the multiplication of any element $ x \in \fm \smallsetminus \fp $ induces an isomorphism on $ E(R/ \fp) $ and hence on $X$. But Since
	$ M/ \fm^j M $ is of finite length for all $ j \in J$, there exists an integer $ n_j $ for which $ x^{n_j}(M/ \fm^j M) = 0 $. Finally since the $\Ext$ functor is linear, it follows that a multiplication of $x$ on $ \Ext^i_R(X , \prod_{j \in J} M/ \fm^j M) $ is both an isomorphism and locally nilpotent for all $ i \geq 0 $. Therefore $ \Ext^i_R(X , \prod_{j \in J} M/ \fm^j M) = 0 $ for all $ i \geq 0 $ and so
	$ \Ext^i_R(X , M) = 0 $ for all $ i \geq 0 $. Consequently, if $ \Ext^n_R(\Hom_R(C,I),M) \neq 0 $, then $ \Ext^n_R(\Hom_R(C,E(R/ \fm)),M) \neq 0 $.  This proves the inequality '$ \leq $'. The inequality '$ \geq $' is trivial.
\end{proof}
The following theorem is a generalization of \cite[Theorem 4.1]{EJ3}.
\begin{thm}\label{A2}
The following are equivalent for an integer $n \geq 0$:
	\begin{itemize}
		\item[(i)]{ $\emph{\id}_R(C) \leq n$.}
		\item[(ii)] {$ C $-$\emph{\cid}_R(M) \leq n$ for any $R$-module $M$.}
		\item[(iii)] {Every $n$-th cosyzygy of an $R$-module is strongly $C$-copure injective.}	
		\item[(iv)] {$ C $-$\emph{\cpd}_R(M) \leq n$ for any $R$-module $M$.}
		\item[(v)] {Every $n$-th syzygy of any $R$-module is strongly $C$-copure projective.}
		\item[(vi)] {$ C $-$\emph{\cfd}_R(M) \leq n$ for any $R$-module $M$.}
		\item[(vii)] {Every $n$-th syzygy of an $R$-module is strongly $C$-copure flat.}	
	\end{itemize}
\end{thm}
\begin{proof}
(i) $\Longrightarrow$ (ii). Assume that $M$ is a non-zero $R$-module. Since $C$ is dualizing, by Lemma 2.5(iii),  we have
$ \fd_R(\Hom_R(C,I)) \leq n $ for any injective $R$-module $I$. Hence, by \cite[Corollary 3.4]{F1}, we have 
$ \pd_R(\Hom_R(C,I)) \leq n $. Thus $\Ext^i_R(\mathcal{I}_C , M) = 0 $ for all $i > n$. So that $ C $-$\cid_R(M) \leq n$.

(ii) $\Longleftrightarrow$ (iii). First assume that $M$ is an $R$-module. Choose an arbitrary coresolution for $M$ by strongly $C$-copure injective modules and let $K_n$ be $n$-th cosyzygy. Use dimension shifting and the fact that $ \Ext_R^{i>0}(\mathcal{I}_C , \mathcal{M}^C_{\sci}) = 0 $, to see that $ \Ext_R^i(\mathcal{I}_C , K_n) \cong \Ext_R^{i+n}(\mathcal{I}_C , M) = 0 $ for all $ i > 0 $. The converse is evident.

(iv) $\Longrightarrow$ (i). By assumption, for any $R$-module $M$ we have we have $ \Tor_i^R(\mathcal{I}_C , M) = 0 $ for all $i > n$. consequently, all modules in $ \mathcal{I}_C $ have finite flat dimensions less than or equal to $n$, whence $\id_R(C) \leq n$.
	
(ii) $\Longrightarrow$ (i).  Is similar to that of (iv) $\Longrightarrow$ (i).

(i) $\Longrightarrow$ (vi). Since $\id_R(C) \leq n$, all modules in $ \mathcal{F}_C $ have finite injective dimensions less than or equal to $n$. Therefore if $M$ is an $R$-module, then $ \Ext^i_R(M , \mathcal{F}_C) = 0 $ for all $ i > n $, whence $ C $-$ \cpd_R(M) \leq n $.

(vi) $\Longrightarrow$ (i). By assumption, $  \Ext^i_R(M , \mathcal{F}_C) = 0 $ for any non-zero $R$-module $ M $ and any $ i > n $. In particular, $ \Ext^i_R(M , C) = 0 $  for any $R$-module $ M $ and any $ i > n $, whence $ \id_R(C) \leq n $.

(iv) $\Longleftrightarrow$ (v) and (vi) $\Longleftrightarrow$ (vii) are similar to that of (ii) $\Longleftrightarrow$ (iii).		
\end{proof}
\begin{cor}
	The following are equivalent:
	\begin{itemize}
		\item[(i)]{ $\emph{\id}_R(C) \leq 1$.}
		\item[(ii)] {Any quotient of a strongly $C$-copure injective $ R $-module is strongly $C$-copure injective.}	
		\item[(iii)] {Any submodule of a strongly $C$-copure projective $ R $-module is strongly $C$-copure projective.}
		\item[(iv)] {Any submodule of a strongly $C$-copure flat $ R $-module is strongly $C$-copure flat.}	
	\end{itemize}
\end{cor}
The following definition can be considered as a relative global dimension of a ring with respect to a semidualizing module. Our aim from this definition, is that to find a necessary and sufficient condition for $C$ to be dualizing (See Corollary 4.30 below).
\begin{defn}
\emph{Set $ C $-$ \cpD(R) = \sup \{ C\emph{-}\cpd_R(M) | M$ is an $R$-module$\}$ and call  $C$-$ \cpD(R) $ the \textit{global $C$-copure projective dimension} of $R$. The notions of $C$-$ \cfD(R) $ and $C$-$ \ciD(R) $ are defined similarly.}
\end{defn}
\begin{cor}
	Assume that $ \emph{\dim}(R) < \infty $. The following are equivalent:
\begin{itemize}
	\item[(i)]{ $C$ is dualizing.}
	\item[(ii)] {$ C $-$\emph{\ciD}(R) < \infty$.}
	\item[(iii)] {$ C $-$\emph{\cpD}(R) < \infty$.}	
	\item[(iv)] {$ C $-$\emph{\cfD}(R) < \infty$.}
\end{itemize}	
Moreover, if one of the above statements holds, then $ \emph{\id}_R(C) =  C\emph{-\ciD}(R) = C\emph{-\cpD}(R) = C\emph{-\cfD}(R) $.
\end{cor}
\begin{proof}
The equivalence of the above statements follows from Theorem 4.27. We only show the equality $ \id_R(C) =  C$-$\ciD(R) $. The other equalities can be proved similarly. If $ \id_R(C) = n < \infty $, then as in the proof of Theorem 4.27, all modules in $ \mathcal{I}_C $ have  finite projective dimensions less than or equal to $n$. Consequently, $ \Ext_R^i(\mathcal{I}_C , M) = 0 $ for any $ i > n $ and all $R$-modules $M$. Therefore, we need only to find an $R$-module $M$ with an injective $R$-module $I$ for which $ \Ext_R^n(\Hom_R(C,I) , M) \neq 0 $. Assume that $E$ is an injective cogenerator. Then
 we have $  \pd_R(\Hom_R(C,E)) \leq n $. But if $  \pd_R(\Hom_R(C,E)) < n $, then $  \fd_R(\Hom_R(C,E)) < n $ and so  $ \id_R(C) < n $, which is impossible. Thus $  \pd_R(\Hom_R(C,E) = n $, and so there exists a non-zero $R$-module $M$ for which $ \Ext_R^n(\Hom_R(C,E) , M) \neq 0 $. This gives the equality $ \id_R(C) =  C$-$\ciD(R) $.
\end{proof}
The following theorem examines the finiteness of relative homological dimensions with respect to a semidualizing module in the local case. We see that the residue field of a local ring is a test module for detecting the finiteness of relative homological dimensions.
\begin{thm}\label{A2}
	Let $(R, \fm , k)$ be local with $ \emph{\dim}(R) = d $. The following are equivalent:
	\begin{itemize}
		\item[(i)]{ $C$ is dualizing.}
		\item[(ii)] {$ C $-$\emph{\cid}_R(M) < \infty$ for any $R$-module $M$.}
		\item[(iii)] {$ C $-$\emph{\cid}_R(k) < \infty$.}	
		\item[(iv)] {$ C $-$\emph{\cfd}_R(M) < \infty$ for any $R$-module $M$.}
		\item[(v)] {$ C $-$\emph{\cfd}_R(k) < \infty$.}	
		\item[(vi)] {$ C $-$\emph{\cpd}_R(M) < \infty$ for any $R$-module $M$.}
		\item[(vii)] {$ C $-$\emph{\cpd}_R(k) < \infty$.}	
	\end{itemize}
Moreover, if one of the above statements holds, then $ d =  C\emph{-\cid}(k) = C\emph{-\cpd}(k) = C\emph{-\cfd}(k) $.	
\end{thm}
\begin{proof}
	Set $ (-)^{\vee} = \Hom_R(-, E(k))$. The parts (i) $ \Longleftrightarrow $ (ii), (i) $ \Longleftrightarrow $ (iv) and (i) $ \Longleftrightarrow $ (vi) are proved in the Theorem 4.27. Also the parts
(ii) $\Longrightarrow$ (iii), (iv) $\Longrightarrow$ (v) and (vi) $\Longrightarrow$ (vii) are evident.

(iii) $ \Longleftrightarrow $ (v). By \cite[Theorem 3.2.1]{EJ1}, we have the isomorphism $ \Ext_R^i(\mathcal{I}_C , k^{\vee}) \cong \Tor_i^R(\mathcal{I}_C , k)^{\vee} $ for all $i \geq 0$. Now the result follows by the fact that $ k^{\vee} \cong k$.

(v) $ \Longleftrightarrow $ (vii). By \cite[Theorem 3.2.1]{EJ1}, we have the isomorphism $ \Ext_R^i \big(k , \big(\mathcal{I}_C\big)^{\vee} \big) \cong \Tor_i^R(\mathcal{I}_C , k)^{\vee} $ and also by \cite[Theorem  3.2.13]{EJ1}, we have the isomorphism $ \Ext_R^i(k , \mathcal{F}_C)^{\vee} \cong \Tor_i^R \big(k , \big(\mathcal{F}_C\big)^{\vee}\big) $ for all $i \geq 0$. Now just note that $ \big(\mathcal{I}_C\big)^{\vee} \subseteq \mathcal{F}_C $ and that $ \big(\mathcal{F}_C\big)^{\vee} \subseteq \mathcal{I}_C $.
	
(v) $\Longrightarrow$ (i). Assume that $ C $-$\cfd_R(k) = n$. Then, in particular, $\Tor^R_i(\Hom_R(C, E(k)) , k) = 0 $ for all $i > n$. Now the isomorphism\\
\centerline{$\Tor^R_i(\Hom_R(C, E(k)) , k) \cong \Ext^i_R(k , C)^{\vee}$}
	shows that $ \Ext^i_R(k , C) = 0 $ for all $ i > n $, whence $C$ is dualizing.
\end{proof}	
The equivalence of the conditions (i), (ii) and (iii) in the following corollary, can be considered as a generalization of \cite[Theorem 2.7]{KY}. Note that $ \cid_R(M) \leq $G-$\id_R(M) $ for any $R$-module $M$.
\begin{cor}\label{A2}
	Let $(R, \fm , k)$ be local with $\emph{\dim}(R) = d $. The following are equivalent:
	\begin{itemize}
		\item[(i)]{ $R$ is Gorenstein.}
		\item[(ii)] {$\emph{\cid}_R(M) < \infty$ for any $R$-module $M$.}
		\item[(iii)] {$\emph{\cid}_R(k) < \infty$.}	
		\item[(iv)] {$\emph{\cfd}_R(M) < \infty$ for any $R$-module $M$.}
		\item[(v)] {$\emph{\cfd}_R(k) < \infty$.}
		\item[(vi)] {$\emph{\cpd}_R(M) < \infty$ for any $R$-module $M$.}
		\item[(vii)] {$\emph{\cpd}_R(k) < \infty$.}	
	\end{itemize}
	Moreover, if one of the above statements holds, then $ d =  \emph{\cid}(k) = \emph{\cpd}(k) = \emph{\cfd}(k) $.
\end{cor}

H.-B. Foxby and A.J. Frankild in \cite[Theorem 4.5]{FF} showed that a Noetherian local ring $R$ is Gorenstein if and only if there exists a cyclic $R$-module of finite G-injective dimension. It is easy to see that $C$ is dualizing if and only if there exists a cyclic $R$-module of finite G$ _C $-injective dimension. Indeed, if $C$ is dualizing, then $ R \ltimes C $ is a Gorenstein local ring by \cite[Theorem 7]{Re}, and hence G$ _C $-$ \id_R(R) =$Gid$ _{R \ltimes C} (R) < \infty$ by \cite[Theorem 2.16]{HJ}. Conversely, if $M$ is a cyclic $R$-module with G$ _C $-$ \id_R(M) < \infty$, then Gid$ _{R \ltimes C} (M) < \infty$ by \cite[Theorem 2.16]{HJ}, and hence $ R \ltimes C $ is a Gorenstein local ring by \cite[Theorem 4.5]{FF}, whence $C$ is dualizing by \cite[Theorem 7]{Re}. Now it is natural to ask

\begin{Ques}
\emph{ Assume that $ (R, \fm) $ is local and there exists a cyclic $R$-module $M$ with $ C $-$ \cid_R(M) < \infty $. Is then $C$ dualizing?}	
\end{Ques}
\section{A Special Case}
In this section, we focus on those semidualizing $R$-modules which have injective dimensions less than or equal to 1. Our aim is to extend a result of E.Enochs and O.Jenda \cite[Theorem 3.3]{EJ5}
\begin{prop}
Let $\emph{\id}_R(C) \leq 1$ and let $M$ be an $R$-module. The following are equivalent:
	\begin{itemize}
		\item[(i)]{ Any $ \mathcal{I}_C $-precover of $M$ is surjective.}
		\item[(ii)] {$M \in \mathcal{M}^C_{\emph{\ci}} $.}	
	\end{itemize}
\end{prop}
\begin{proof}
(i) $\Longrightarrow$ (ii).  By the same argument as in the proof of Theorem 4.27, we see that all modules in $ \mathcal{I}_C $ have finite projective dimension less than or equal to 1. By assumption, there exists an exact sequence \\
\centerline{$ 0 \rightarrow K \rightarrow \Hom_R(C,E) \rightarrow M \rightarrow 0 $, $ (*) $}
in which $ E $ is injective and $ K = \ker (\Hom_R(C,E) \rightarrow M) $. Application of the functor $ \Hom_R(\mathcal{I}_C ,-) $ on $ (*) $ yields an exact sequence \\
\centerline{$ \Ext_R^1(\mathcal{I}_C , \Hom_R(C,E))  \rightarrow \Ext_R^1(\mathcal{I}_C , M) \rightarrow \Ext_R^2(\mathcal{I}_C , K) = 0 $.}
Observe that $  \Ext_R^1(\mathcal{I}_C , \Hom_R(C,E)) \cong \Hom_R(\Tor^R_1(\mathcal{I}_C , C) , E) = 0 $ by \cite[Theorem 3.2.1]{EJ1}. It follows that $ \Ext_R^1(\mathcal{I}_C , M) = 0$, whence $M \in \mathcal{M}^C_{\ci} $.

(ii) $\Longrightarrow$ (i). Assume that $M \in \mathcal{M}^C_{\ci} $. Clearly, $M$ is a homomorphic image of a projective module, say $P$.  Now the $\mathcal{I}_C$-preenvelope of $P$ is injective by Lemma 2.10(ii), since
$ P \in \mathcal{A}_C(R) $. Consider the following push-out diagram:
 \begin{displaymath}
 \xymatrix{
 	&&  0  \ar[d] &  0  \ar[d]  \\
 	0 \ar[r] & K \ar[r] \ar@{=}[d] & P \ar[r] \ar[d] & M \ar[r] \ar[d] & 0 \\
 	0 \ar[r] & K \ar[r] & \Hom_R(C,E) \ar[r] \ar[d] & D \ar[r] \ar[d] & 0 \\
 	&&  X \ar@{=}[r] \ar[d] & X \ar[d] \\
 	&&  0 &  0             \\
 }
 \end{displaymath}
 where $K = \ker (P \rightarrow M) $ and $ X = \coker (P \rightarrow \Hom_R(C,E))$. Note that $ X \in \mathcal{A}_C(R) $ since the middle column is exact and that both $P$ and $ \Hom_R(C,E) $ are in $ \mathcal{A}_C(R) $. Therefore $ \Tor_1^R(C,X)  = 0$. Hence applying the functor $ C \otimes_R - $ on the the middle column, we get the exact sequence \\
 \centerline{$ 0 \rightarrow C \otimes_R P  \rightarrow E \rightarrow C \otimes_R X \rightarrow 0 $.}
 Next, note that $ \id_R(C \otimes_R P) \leq 1 $, since $P$ is projective and that $ \id_R(C) \leq 1 $. Thus $ C \otimes_R X $ must be injective since it is a first cosyzygy of $ C \otimes_R P $. So that $ X \in \mathcal{I}_C $ by Lemma 2.5(i). It follows that
 $ \Ext_R^1(X , M) = 0 $ since $M$ is $C$-copure injective. Therefore the right hand side column is split, and so there exists an epimorphism
 $ \Hom_R(C , E) \rightarrow M $. Consequently, any $\mathcal{I}_C $-precover of $M$ must be surjective.
\end{proof}
\begin{cor}
	Let $\emph{\id}_R(C) \leq 1$ and let $M$ be an $R$-module. The following are equivalent:
	\begin{itemize}
		\item[(i)]{ Any $ \mathcal{F}_C $-preenvelope of $M$ is injective.}
		\item[(ii)] {$M \in \mathcal{M}^C_{\emph{\cf}} $.}	
	\end{itemize}
\end{cor}
\begin{proof}
	 Assuma that $E$ is an injective cogenerator and that $ (-)^{\vee} = \Hom_R(-,E) $.
	
(i) $\Longrightarrow$ (ii). Suppose that $ M \rightarrow C \otimes_R F $ is a $ \mathcal{F}_C $-preenvelope of $M$. Then $ (C \otimes_R F)^{\vee} \rightarrow  M^{\vee} $	is surjective. No since $ (C \otimes_R F)^{\vee} \cong \Hom_R(C, F^{\vee}) \in \mathcal{I}_C $, it follows that any $ \mathcal{I}_C $-precover of $ M^{\vee} $ is surjective. Therefore, by Proposition 5.1, $ M^{\vee} \in \mathcal{M}^C_{\ci}$, whence
$ M \in \mathcal{M}^C_{\cf}$ by Lemma 4.5(ii).

(ii) $\Longrightarrow$ (i). By Lemma 4.5(ii), $ M^{\vee} \in \mathcal{M}^C_{\ci}$, and then any $ \mathcal{I}_C $-precover of $M^{\vee}$ is surjective by Proposition 5.1. Suppose that $ \Hom_R(C,I) \rightarrow M^{\vee} $ is a $ \mathcal{I}_C $-precover of $M^{\vee}$. Then we have an injection
$  M^{\vee \vee} \rightarrow \Hom_R(C,I)^{\vee} $. But since $ \Hom_R(C,I)^{\vee} \cong C \otimes I^{\vee} \in \mathcal{F}_C  $ by \cite[Theorem 3.2.11]{EJ1}, and that $ M \hookrightarrow M^{\vee \vee} $, we see that any $ \mathcal{F}_C $-preenvelope of $M$ is injective.
\end{proof}
The following proposition is dual to the Proposition 5.1. Note that, in general, the class $ \mathcal{P}_C  $ is not preenveloping.
\begin{prop}
	Let $\emph{\id}_R(C) \leq 1$ and let $M$ be an $R$-module. The following are equivalent:
	\begin{itemize}
		\item[(i)]{$M$ can be embedded in a $C$-projective $R$-module.}
		\item[(ii)] {$M \in \mathcal{M}^C_{\emph{\cp}} $.}	
	\end{itemize}
\end{prop}
\begin{proof}
	(i) $\Longrightarrow$ (ii).  Note that $ \id_R(C \otimes_R P) \leq 1 $ for any projective module $P$ Since $ \id_R(C) \leq 1 $. Consequently, all modules in $ \mathcal{P}_C $ have finite injective dimension less than or equal to 1.  By assumption, there exists an exact sequence \\
	\centerline{$ 0 \rightarrow M \rightarrow C \otimes_R P \rightarrow K \rightarrow 0 $, $ (\dagger) $}
	in which $ P $ is projective and $ K = \coker (M \rightarrow C \otimes_R P) $. Application of the functor $ \Hom_R(- ,\mathcal{F}_C) $
	on $ (\dagger) $ yields an exact sequence \\
	\centerline{$ \Ext_R^1( C \otimes_R P, \mathcal{F}_C)  \rightarrow \Ext_R^1(M, \mathcal{F}_C) \rightarrow \Ext_R^2(K , \mathcal{F}_C) = 0 $.}
	Observe that $  \Ext_R^1(C \otimes_R P, \mathcal{F}_C) \cong \Hom_R(P, \Ext_R^1(C , \mathcal{F}_C)) = 0 $. Therefore
	 $ \Ext_R^1(M,\mathcal{F}_C ) = 0$, whence $M \in \mathcal{M}^C_{\cp} $.
	
	(ii) $\Longrightarrow$ (i). Assume that $M \in \mathcal{M}^C_{\cp} $. Clearly, $M$ can be embedded in an injective module, say $E$.  Now the $\mathcal{P}_C$-precover of $E$ is surjective by Lemma 2.10(ii), since
	$ E \in \mathcal{B}_C(R) $. Consider the following pull-back diagram:
	 \begin{displaymath}
	 \xymatrix{
	 	&  0  \ar[d] &  0  \ar[d]  \\
	 	&  N \ar@{=}[r] \ar[d] & N \ar[d] \\
	 	0 \ar[r] & D \ar[r] \ar[d] & C \otimes_R P \ar[r] \ar[d] & X \ar[r] \ar@{=}[d] & 0 \\
	 	0 \ar[r] & M \ar[r] \ar[d] & E \ar[r] \ar[d]             & X \ar[r]   & 0 \\
	           	&  0             & 0  \\}
	 \end{displaymath}
	where $N = \ker (C \otimes_R P \rightarrow E) $ and $ X = \coker (M \rightarrow E)$. Note that $ N \in \mathcal{B}_C(R) $ since the middle column is exact and that both $E$ and $ C \otimes_R P $ are in $ \mathcal{B}_C(R) $. Therefore $ \Ext^1_R(C,N)  = 0$. Hence applying the functor $ \Hom_R(C,-) $ on the the middle column, we get the exact sequence \\
	\centerline{$ 0 \rightarrow \Hom_R(C,N)  \rightarrow P \rightarrow \Hom_R(C, E) \rightarrow 0 $.}
	Next, note that $ \pd_R(\Hom_R(C,E)) \leq 1 $ by \cite[Corollary 3.4]{F1}, and hence $ \Hom_R(C,N) $ must be projective since it is a first syzygy of $ \Hom_R(C,E) $. So that $ N \in \mathcal{P}_C $ by Lemma 2.5(ii). It follows that
	$ \Ext_R^1(M , N) = 0 $ since $M \in \mathcal{M}^C_{\cp} $. Therefore the left hand side column is split, and so there exists an embedding
	$ M \hookrightarrow C \otimes_R P $.
\end{proof}
The following is a generalization of \cite[Theorem 2.5]{EJ5}.
\begin{cor}
Let $ \emph{\id}_R(C) \leq 1 $ and let $M \in \mathcal{A}_C(R)$. Then any $\mathcal{I}_C $-precover of $\emph{\Ext}_R^1(F,M) $ is surjective, for any flat $R$-module $F$.
\end{cor}
\begin{proof}
Since $M \in \mathcal{A}_C(R)$, every $ \mathcal{I}_C $-preenvelope of $M$ is injective by Lemma 2.10(ii). Hence there is an exact sequence \\
	 \centerline{$ 0 \rightarrow M  \rightarrow \Hom_R(C,E) \rightarrow L \rightarrow 0 $, $ (*) $}
	 in which $ E $ is injective and $ L = \coker (M  \rightarrow \Hom_R(C,E)) $. Now any $\mathcal{I}_C $-precover of $L$ is surjective, and hence $L$ is $C$-copure injective by Proposition 5.1. Apply the functor $ \Hom_R(F,-) $ to the exact sequence $ (*) $ to induce
	 an exact sequence \\
	  \centerline{$ \Hom_R(F,L)  \rightarrow \Ext^1_R(F,M) \rightarrow \Ext^1_R(F,\Hom_R(C,E)) = 0 $.}
	  Next consider the exact sequence \\
	  \centerline{$0 \rightarrow K \rightarrow \Hom_R(F,L)  \rightarrow \Ext^1_R(F,M) \rightarrow 0 $,}
	  in which $ K = \ker (\Hom_R(F,L)  \rightarrow \Ext^1_R(F,M)) $, to induce an exact sequence \\
	   \centerline{$ 0 = \Ext^1_R(\mathcal{I}_C ,\Hom_R(F,L))  \rightarrow \Ext^1_R(\mathcal{I}_C,\Ext^1_R(F,M))  \rightarrow \Ext^2_R(\mathcal{I}_C,K) $.}
	   But all modules in $ \mathcal{I}_C $ have finite projective dimensions less than or equal to 1. Hence  $ \Ext^2_R(\mathcal{I}_C,K) = 0$ and so $ \Ext^1_R(\mathcal{I}_C,\Ext^1_R(F,M)) = 0 $. Therefore $ \Ext^1_R(F,M) $ is $C$-copure injective, and thus we are done by Proposition 5.1. 	
\end{proof}
The following theorem is a generalization of \cite[theorem 3.3]{EJ5}.
\begin{thm}
Let $ \emph{\id}_R(C) \leq 1 $ and let $M \in \mathcal{A}_C(R)$. The following are equivalent:
	\begin{itemize}
		\item[(i)]{$ M \in \mathcal{M}_{\emph{\Cot}}$.}
		\item[(ii)]{$ C \otimes_R M \in \mathcal{M}^C_{\emph{\Cot}}$.}
		\item[(iii)]{$ \emph{\Ext}^1_R(E(C), C \otimes_R M) = 0 $.}	
	\end{itemize}
\end{thm}
\begin{proof}
(i) $\Longleftrightarrow$ (ii).	Let $F$ be a flat $R$-module. On has the isomorphisms 
   \[\begin{array}{rl}
   \Ext^i_R(F,M) &\cong \Ext^i_R(\Hom_R(C , C \otimes_R F), \Hom_R(C , C \otimes_R M)) \\
   &\cong \Ext^i_{\mathcal{P}_C}(C \otimes_R F, C \otimes_R M)\\
   &\cong \Ext^i_R(C \otimes_R F, C \otimes_R M),\\
   \end{array}\]
in which the first isomorphism holds because $ M $ and $F$ are in $  \mathcal{A}_C(R) $, the second isomorphism is from \cite[Theorem 4.1]{TW}, and the last isomorphism is from \cite[Corollary 4.2(a)]{TW} since $ C \otimes_R M \in \mathcal{B}_C(R) $ by \cite[Theorem 2.8(b)]{TW}. Hence $ M \in \mathcal{M}_{\Cot} $ if and only if $C \otimes_R M \in \mathcal{M}^C_{\Cot}$.
	
 (ii) $\Longrightarrow$ (iii). Just note that by \cite[Theorem 3.3]{RT}, $ E(C) $ is $C$-flat.

  (iii) $\Longrightarrow$ (ii).We have to prove that $ \Ext^1_R(\mathcal{F}_C, C \otimes_R M) = 0 $. Assume that $ F $ is a flat $R$-module. Consider the exact sequence \\
 \centerline{$ 0  \rightarrow C \otimes_R F \rightarrow E(C \otimes_R F) \rightarrow L \rightarrow 0 $, $ (*) $}
 in which $ L = \coker(C \otimes_R F \rightarrow E(C \otimes_R F)) $. By \cite[theorem 3.3]{RT}, $ E(C \otimes_R F) $ is $C$-flat. Hence $\pd_R(\Hom_R(C,L)) = C$-$\fd_R(L) \leq 1$ by Lemma 2.5(iii), and then \cite[Corollary 3.4]{F1} implies that $\pd_R(\Hom_R(C,L)) \leq 1$. Also $ C$-$\fd_R(L) \leq 1 $ implies that $ L \in \mathcal{B}_C(R) $. Note that $ C \otimes_R  M \in \mathcal{B}_C(R)$ by \cite[Theorem 2.8]{TW}. Now we have
   \[\begin{array}{rl}
   \Ext^2_R(L, C \otimes_R M) &\cong \Ext^2_{\mathcal{P}_C}(L, C \otimes_R M) \\
   &\cong \Ext^2_R(\Hom_R(C,L), \Hom_R(C,C \otimes_R M))\\
   &\cong \Ext^2_R(\Hom_R(C,L), M)\\
   &= 0,\\
   \end{array}\]
where the first isomorphism is from \cite[Corollary 4.2(a)]{TW}, the second isomorphism  is from \cite[Theorem 4.1]{TW} and the third one holds because $M \in \mathcal{A}_C(R)$. Therefore application of the functor $ \Hom_R( - , C \otimes_R M) $ on the exact sequence $ (*) $ yields an exact sequence \\
 \centerline{$ \Ext^1_R(E(C \otimes_R F) , C \otimes_R M) \rightarrow \Ext^1_R(C \otimes_R F , C \otimes_R M) \rightarrow  0 $.}
 According to \cite[Theorem 23.2]{M}, we have $ \Ass_R(C \otimes_R F) \subseteq \Ass_R(C) $. It follows that $ E(C \otimes_R F) = \underset{\fp \in \mathcal{T}} \oplus E(R/ \fp) $, with $ \mathcal{T} \subseteq \Ass_R(C) $. Hence $ \Ext^1_R(E(C), C \otimes_R M) = 0 $ implies that
 $ \Ext^1_R(E(C \otimes_R F) , C \otimes_R M) = 0 $. Thus $ \Ext^1_R(C \otimes_R F , C \otimes_R M) = 0 $, as wanted.
\end{proof}

\bibliographystyle{amsplain}

\end{document}